\documentclass[11pt]{article}
\usepackage{amsmath,amsfonts,amssymb,amsthm,enumerate,graphicx}
\usepackage{tikz,authblk}
\usepackage{subfig}
\usepackage{url}
\usepackage{bm}

\newtheorem{theorem}{{\bf Theorem}}[section]

\newtheorem{corollary}[theorem]{{\bf Corollary}}
\newtheorem{definition}[theorem]{{\bf Definition}}

\newtheorem{lemma}[theorem]{{\bf Lemma}}

\newtheorem{proposition}[theorem]{{\bf Proposition}}
\newtheorem{remark}[theorem]{{\bf Remark}}

\begin{document}


\author[1] {Basudeb Datta}
\author[2] { Dipendu Maity}

\affil[1]{Department of Mathematics, Indian Institute of Science, Bangalore 560\,012, India.
 dattab@iisc.ac.in, bdatta17@gmail.com.}

\affil[2]{Department of Sciences and Mathematics,
Indian Institute of Information Technology Guwahati,
Bongora, Assam 781\,015, India.
dipendu@iiitg.ac.in.}

\title{Platonic solids, Archimedean solids and semi-equivelar maps on the sphere}


\date{To appear in `{\bf Discrete Mathematics}'}

\maketitle

\vspace{-10mm}

\begin{abstract}
A map $X$ on a surface is called vertex-transitive if the automorphism group of $X$ acts transitively on the set of vertices of $X$.  A map is called semi-equivelar if the cyclic arrangement of faces around each vertex is same. In general, semi-equivelar maps on a surface form a bigger class than vertex-transitive maps. There are semi-equivelar maps on the torus, the Klein bottle and other surfaces which are not vertex-transitive.  

It is known that the boundaries of  Platonic solids, Archimedean solids,  regular prisms and  anti-prisms are vertex-transitive maps on $\mathbb{S}^2$. Here we show that there is exactly one semi-equivelar map on $\mathbb{S}^2$ which is not vertex-transitive.  As a consequence, we show that all the semi-equivelar maps on $\mathbb{RP}^2$ are vertex-transitive. Moreover, every semi-equivelar map on $\mathbb{S}^2$ can be geometrized, i.e., every semi-equivelar map on $\mathbb{S}^2$ is isomorphic to a semi-regular tiling of $\mathbb{S}^2$. In the course of the proof of our main result, we present a combinatorial characterisation in terms of  an inequality of all the types of semi-equivelar maps on $\mathbb{S}^2$. Here we present combinatorial proofs of all the results. 

\end{abstract}

\noindent {\small {\em MSC 2020\,:} 52C20, 52B70, 51M20, 57M60.

}


\section{Introduction} \label{sec:introduction}

By a map we mean a polyhedral map on a surface. So, a face of a map is an $n$-gon for some $n\geq 3$ and two intersecting  faces intersect either on a vertex or on an edge.  A map on a surface is also called a {\em topological tiling} of the surface. If all the faces of a map are triangles then the map is called {\em simplicial}. Maps on the sphere and the torus are called {\em spherical} and {\em toroidal} maps respectively.  A map $X$ is said to be  {\em vertex-transitive} if the automorphism group $\mbox{Aut}(X)$ of $X$ acts transitively on the set  $V(X)$ of vertices of $X$. In \cite{lutz1999}, Lutz found all the (77 in numbers) vertex-transitive simplicial maps  with at most $15$ vertices.

For a vertex $u$ in a map $X$, the faces containing $u$ form a cycle (called the {\em face-cycle} at $u$) $C_u$ in the dual graph  of $X$. So, $C_u$ is of the form $(F_{1,1}\mbox{-}\cdots \mbox{-}F_{1,n_1})\mbox{-}\cdots\mbox{-}(F_{k,1}\mbox{-}\cdots \mbox{-}F_{k,n_k})\mbox{-}F_{1,1}$, where $F_{i,j}$  is a  $p_i$-gon for $1\leq j \leq n_i$ and  $p_i\neq p_{i+1}$  for $1\leq i\leq k$  (addition in the suffix is modulo $k$).  A map $X$ is called {\em semi-equivelar} if the cyclic arrangement of faces around each vertex is same. More precisely,  there exist integers $p_1, \dots, p_k\geq 3$ and $n_1, \dots, n_k\geq 1$, $p_i\neq p_{i+1}$, such that $C_u$ is of the form as above for all $u\in V(X)$. In such a case, $X$ is called a {\em semi-equivelar map of  vertex-type} $[p_1^{n_1}, \dots, p_k^{n_k}]$ (or, a {\em map of type} $[p_1^{n_1}, \dots, p_k^{n_k}]$). (We identify a cyclic tuple $[p_1^{n_1}, p_2^{n_2}, \dots, p_k^{n_k}]$ with $[p_k^{n_k}, \dots, p_2^{n_2}, p_1^{n_1}]$ and with $[p_2^{n_2}, \dots, p_k^{n_k}, p_1^{n_1}]$.) Clearly, vertex-transitive maps are semi-equivelar.

A {\em semi-regular tiling} of a surface $S$ of constant curvature (\textit{eg.}, the round sphere, the Euclidean plane or the hyperbolic plane) is a semi-equivelar map on $S$ in which each face  is a regular polygon and each edge is a geodesic. 
An {\em Archimedean tiling} of the plane $\mathbb{R}^2$ is a semi-regular tiling of the Euclidean plane. 
There are eleven types of semi-equivelar toroidal maps and all these are quotients of Archimedean tilings of the plane (\cite{DM2017}, \cite{DM2018}). Among these 11 types, 4 types of maps are always  vertex-transitive and there are infinitely many such examples in each type (\cite{Ba1991}, \cite{DM2017}). For each of the other seven types, there exists a semi-equivelar toroidal map  which is not vertex-transitive (\cite{DM2017}). Although, there are vertex-transitive maps of each of these seven types also (\cite{Ba1991}, \cite{Su2011t}, \cite{Th1991}). Similar results are known for the Klein bottle (\cite{Ba1991}, \cite{DM2017}, \cite{Su2011kb}). If the Euler characteristic $\chi(M)$ of a surface $M$ is negative then the number of semi-equivelar maps on $M$ is finite and at most $-84\chi(M)$ (\cite{Ba1991}).
Nine examples of non-vertex-transitive semi-equivelar maps on the surface of Euler characteristic $-1$ are known (\cite{TU2017}). There are exactly three non vertex-transitive semi-equivelar  simplicial maps on the double torus  (\cite{DU2006}). 

It follows from the results in \cite{DG2021} that there exist semi-regular tilings of the hyperbolic plane of infinitely many different vertex-types. It is also shown that there exists a unique semi-regular tiling of the hyperbolic plane of vertex-type $[p^q]$ for each pair $(p, q)$ of positive integers satisfying $1/p+1/q<1/2$. Moreover, these tilings are vertex-transitive. 
It follows from the results in \cite{GS1978} and \cite{M2020} that there exists a semi-equivelar map on the plane of vertex-type $[3^1, p^3]$ for each $p\geq 5$ odd and there does not exist any vertex-transitive map on the plane of vertex-type $[3^1,q^3]$, for each $q \equiv \pm 1$ (mod 6). 

All vertex-transitive spherical maps  are known. These are the boundaries of Platonic solids, Archimedean solids and two infinite families (\cite{Ba1991}, \cite{GS1981}). Other than these, there exists a semi-equivelar spherical map, namely, the boundary of the pseudorhombicuboctahedron (\cite{Gr2009}, \cite{wiki}). It is known that quotients of ten centrally symmetric vertex-transitive spherical maps (namely, the boundaries of icosahedron, dodecahedron and eight Archimedean solids) are all the vertex-transitive maps on the real projective plane $\mathbb{RP}^2$ (\cite{Ba1991}). Here we show that these are also all the semi-equivelar maps on $\mathbb{RP}^2$.  We prove

\begin{theorem}  \label{thm:s2}
Let $X$ be a semi-equivelar spherical map. Then, up to isomorphism, $X$ is the boundary of a Platonic solid,   an Archimedean solid, a regular prism, an antiprism or the pseudorhombicuboctahedron.
\end{theorem}

\begin{theorem}  \label{thm:rp2}
If $Y$ is a semi-equivelar map on $\mathbb{RP}^2$ then  the vertex-type of $Y$ is $[5^3]$, $[3^5],$ $[4^1, 6^2],$ $[3^1, 5^1, 3^1, 5^1], [3^1, 4^3],$ $[4^1, 6^1, 8^1], [3^1, 4^1, 5^1, 4^1], [4^1, 6^1, 10^1], [3^1, 10^2]$ or $[5^1, 6^2]$. Moreover, in each case, there exists a unique semi-equivelar map on $\mathbb{RP}^2$.
\end{theorem}


\begin{corollary} \label{cor:s2vt}
The boundary of the pseudorhombicuboctahedron is not vertex-transitive and all the other semi-equivelar spherical maps are vertex-transitive.
\end{corollary}

As consequences we get 

\begin{corollary} \label{cor:tiling}
$(a)$ Each semi-equivelar spherical map is isomorphic to a semi-regular tiling of $\mathbb{S}^2$. $(b)$ Each semi-equivelar map on $\mathbb{RP}^2$ is isomorphic to a semi-regular tiling of $\mathbb{RP}^2$. 
\end{corollary} 

\begin{corollary} \label{cor:rp2vt}
All the semi-equivelar maps on $\mathbb{RP}^2$ are vertex-transitive.  
\end{corollary}


\section{Examples of twenty two face-regular polyhedra} \label{sec:example}

Here we present twenty two  classically known $3$-polytopes.


\begin{figure}[htb]
\begin{center}
\includegraphics[scale=.5]{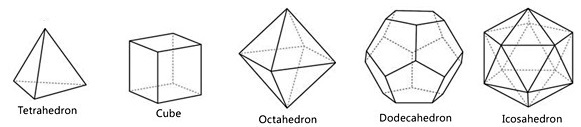}
\caption{Platonic Solids (from \cite{mathfun}).} \label{fig:platonic}
\end{center}
\end{figure}

\vspace{-5mm}

\begin{figure}[htb]
\begin{center}
\includegraphics[scale=.52]{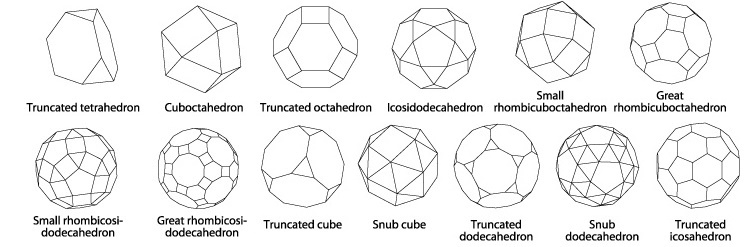}
\caption{Archimedean Solids (from \cite{mathfun}).} \label{fig:archi}
\end{center}
\end{figure}

\vspace{-5mm}

\begin{figure}[htb]
\begin{center}
\includegraphics[scale=.25]{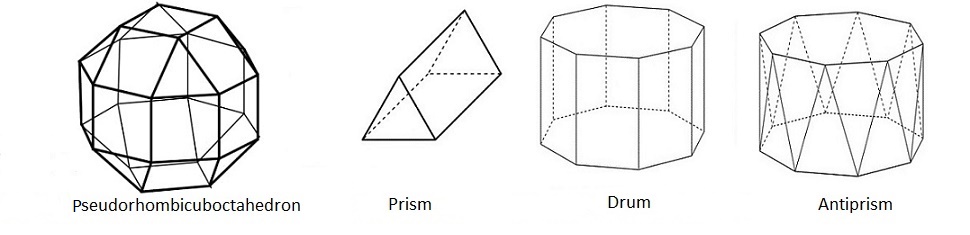}
\caption{Pseudorhombicuboctahedron, Prisms $P_3$ $\&$ $P_8$, Antiprism $Q_8$ (from \cite{Gr2009}, \cite{mathetc}).} \label{fig:prism}
\end{center}
\end{figure}

\vspace{-3mm}

If $P$ is one of the first nineteen of above twenty two  polytopes  then (i) all the vertices of $P$ are points on a  sphere whose centre is same as the centre of $P$, (ii) each 2-face of $P$ is a regular polygon, and (iii) the cyclic arrangement of faces around each vertex is same (equivalently, the boundary complex $\partial P$ of $P$ is a semi-equivelar map). Without loss of generality, we assume that (iv) the vertices of $P$ are points on the unit 2-sphere $\mathbb{S}^{\hspace{.2mm}2}$ with centre $(0,0,0)$. We say a 3-polytope is {\em face-regular} if it satisfies (ii).

For $n\geq 3$, let $P_n$ be the polytope whose  vertex-set  is 
\begin{align*}
\left\{(1+\sin^2\frac{\pi}{n})^{-\frac{1}{2}}\left(\cos\frac{2m\pi}{n}, \sin\frac{2m\pi}{n}, 
\pm\sin\frac{\pi}{n}  \right) : 0\leq m\leq n-1\right\}. 
\end{align*}
The polytope  $P_4$ is a cube and, for $n\neq 4$, the boundary of $P_n$ consists of $n$ squares and two regular $n$-gons. This polytope $P_n$ is called a $2n$-vertex  {\em regular prism} or {\em drum} or  {\em ladder}. We have chosen the vertices  of $P_n$ on the unit sphere $\mathbb{S}^{\hspace{.2mm}2}$ and coordinates are such that  $P_n$ satisfies (i), \dots, (iv) above.

For $n\geq 3$, let $Q_n$ be the polytope whose  vertex-set  is 

\begin{align*}
&\left\{\!(\sin^2\frac{\pi}{n} + \cos^2\frac{\pi}{2n} )^{-\frac{1}{2}}\!\left(\cos\frac{(2m+1)\pi}{n}, 
\sin\frac{(2m+1)\pi}{n}, (\sin^2\frac{\pi}{n} -\sin^2\frac{\pi}{2n})^{\frac{1}{2}} \right),\right.  \\ 
&\quad \left.(\sin^2\frac{\pi}{n} + \cos^2\frac{\pi}{2n} )^{-\frac{1}{2}}\!\left(\cos\frac{2m\pi}{n}, \sin\frac{2m\pi}{n}, - (\sin^2\frac{\pi}{n} -\sin^2\frac{\pi}{2n})^{\frac{1}{2}}\right)  : 0\leq m\leq n-1\right\}.  
\end{align*}
The polytope $Q_3$ is an octahedron and, for $n\geq 4$, the boundary of $Q_n$ consists of $2n$ equilateral triangles  and two regular $n$-gons.  This polytope $Q_n$ is called a $2n$-vertex {\em antiprism}. Moreover, with the chosen coordinates of vertices,  $Q_n$ satisfies  (i), \dots,  (iv) above. 

\smallskip

A  {\em truncation} is an operation that cuts polytope vertices, creating a new face in place of each vertex. 
A {\em rectification} or {\em complete-truncation} is the process of truncating a polytope by marking the midpoints of all its edges, and cutting off its vertices at those points. Thus, a rectification truncates edges to points. For more on truncation  see \cite[Chapter 8]{Cox1973}. In Section \ref{sec:proofs-2}, we present some combinatorial versions of truncation and rectification of polytopes. Proposition \ref{prop:cox} gives some relations among the five Platonic solids and the thirteen 
Archimedean solids in terms of truncation and rectification.

\begin{proposition}[Coxeter]\label{prop:cox}
{\rm The truncation of tetrahedron (respectively, cube, octahedron, dodecahedron, icosahedron, cuboctahedron and icosidodecahedron) gives the truncated tetrahedron (respectively, truncated cube, truncated octahedron, truncated dodecahedron, truncated icosahedron, great rhombicuboctahedron and great rhombicosidodecahedron). The rectification of cube (respectively, dodecahedron, icosidodecahedron and cuboctahedron) gives the cuboctahedron (respectively, icosidodecahedron,   small rhombicosidodecahedron and small rhombicuboctahedron).}
\end{proposition} 

\begin{remark} \label{remark:regular}
{\rm We all know that a 3-polytope is called {\em regular} if its automorphism group acts transitively on the set of flags.  But, there are confusions about the names `semi-regular polyhedra',  `uniform polyhedra' and `Archimedean polyhedra' (\cite{Gr2009}). Agreeing with most of the  authors in the literature, we  write the following. A 3-polytope is  {\em semi-regular} (or {\em uniform}) if its automorphism group acts transitively on the set of vertices. A 3-polytope is an {\em Archimedean polyhedron}  if it satisfies  (ii) and (iii)  above (\cite{Gr2009, Ha1996}). So, a semi-regular polyhedron is an Archimedean polyhedron and  the pseudorhombicuboctahedron is an Archimedean polyhedron which is not a semi-regular polyhedron. Observe that  all the twenty two polyhedra presented above are Archimedean polyhedra. 
}
\end{remark}


\section{Classification of vertex-types of semi-equivelar maps on $\mathbb{S}^{\hspace{.2mm}2}$} \label{sec:proofs-1}

In this section, we present a combinatorial characterisation  of all the types of semi-equivelar maps on $\mathbb{S}^{\hspace{.2mm}2}$ in terms of  an inequality. We need this in Section \ref{sec:proofs-2} to prove Theorem \ref{thm:s2}. 

\smallskip

Let $F_1\mbox{-}\cdots\mbox{-}F_m\mbox{-}F_1$ be the face-cycle of a vertex $u$ in a map. Then $F_i \cap F_j$ is either $u$ or an edge through $u$. Thus, the face $F_i$ must be of the form $u_{i+1}\mbox{-}u\mbox{-}u_i\mbox{-}P_i\mbox{-}u_{i+1}$, where $P_i = \emptyset$ or a path and $P_i \cap P_j = \emptyset$ for $i \neq j$. Here addition in the suffix is modulo $m$. So, $u_{1}\mbox{-}P_1\mbox{-}u_2\mbox{-}\cdots\mbox{-}u_m\mbox{-}P_m\mbox{-}u_1$ is a cycle and said to be the {\em link-cycle} of $u$. For a simplicial complex, $P_i = \emptyset$ for all $i$, and the link-cycle of a vertex is the link of that vertex.

A face in a map of the form $u_1\mbox{-}u_2\mbox{-}\cdots\mbox{-}u_n\mbox{-}u_1$ is also denoted by $u_1u_2\cdots u_n$. The faces with 3, 4, \dots, 10 vertices are called {\em triangle}, {\em square}, \dots, {\em decagon} respectively.

If $X$ is the boundary of a Platonic solid, an Archimedean solid, the pseudorhombicuboctahedron, a regular prism or an antiprism then the vertex-type of $X$ is one of the cycle tuples of the following set.
\begin{align} \label{a-sum<2}
{\mathcal A}  := &\left\{[3^3], [3^4], [4^3], [3^5], [5^3], [3^4, 5^1],  [3^4, 4^1], [3^1, 5^1, 3^1, 5^1], [3^1, 4^1, 3^1, 4^1],  \right. \nonumber \\
& \qquad [3^1, 4^1, 5^1, 4^1], [3^1, 4^3], [5^1, 6^2], [4^1, 6^1, 8^1], [4^1, 6^1, 10^1],  [4^1, 6^2],  \nonumber \\
&\qquad \left. [3^1, 6^2], [3^1, 8^2], [3^1, 10^2], [3^1, 4^2]\}  \cup \{[4^2, r^1], [3^3, s^1], \,  r \geq 5,  s \geq 4\right\}.
\end{align}
If $[p_1^{n_1}, \dots, p_{\ell}^{n_{\ell}}]\in \mathcal{A}$ then one readily checks that $\sum\limits_{i=1}^{\ell}\frac{n_i(p_i-2)}{p_i} < 2$. Here we prove the following converse.

\begin{theorem} \label{thm:inequality}
Let $X$ be a semi-equivelar map of vertex-type  $[p_1^{n_1}, \dots, $ $p_{\ell}^{n_{\ell}}]$ on a $2$-manifold.  If $\sum\limits_{i=1}^{\ell}\frac{n_i(p_i-2)}{p_i} < 2$ then $[p_1^{n_1}, \dots, p_{\ell}^{n_{\ell}}]\in \mathcal{A}$.
\end{theorem} 

We need the following technical lemma of \cite{DM2020} to prove Theorem \ref{thm:inequality}.

\begin{lemma} [Datta $\&$ Maity] \label{DM2020}
If $[p_1^{n_1}, \dots, p_k^{n_k}]$ satisfies any of the following three properties then $[p_1^{n_1}$, $\dots, p_k^{n_k}]$ cannot be the vertex-type of any semi-equivelar map on a surface.
\begin{enumerate}[{\rm (i)}]
\item There exists $i$ such that $n_i=2$, $p_i$ is odd and $p_j\neq p_i$ for all $j\neq i$.

\item There exists $i$ such that $n_i=1$, $p_i$ is odd, $p_j\neq p_i$ for all $j\neq i$ and $p_{i-1}\neq p_{i+1}$. (Here, addition in the subscripts are modulo $k$.)

\item $[p_1^{n_1}, \dots, p_k^{n_k}]$ is of the form $[p^1, q^m, p^1, r^n]$, where $p, q, r$ are distinct and $p$ is odd.
\end{enumerate}
\end{lemma}

\begin{proof}[Proof of Theorem \ref{thm:inequality}]
Let $d$ be the degree of each vertex in $X$.  Consider the $k$-tuple $(q_1^{m_1}, \dots, q_k^{m_k})$, where $3 \le q_1 < \dots <q_k$ and, for each $i=1, \dots,k$, $q_i = p_j$ for some $j$, $m_i = \sum_{p_j = q_i}n_j$. So, $\sum_{i=1}^k  m_i = \sum_{j=1}^{\ell}n_j = d$ and $\sum_{i=1}^k \frac{m_i}{q_i} = \sum_{j=1}^{\ell}\frac{n_j}{p_j}$. 
Thus, 
\begin{align} \label{eq:2}
2> \sum\limits_{j=1}^{\ell}\frac{n_j(p_j-2)}{p_j} = \sum\limits_{j=1}^{\ell}  n_j - 2\sum\limits_{j=1}^{\ell} \frac{n_j}{p_j} =
\sum\limits_{i=1}^k  m_i - 2\sum\limits_{i=1}^k \frac{m_i}{q_i}  = d -2\sum\limits_{i=1}^k \frac{m_i}{q_i}. 
\end{align}
So, $d-2 < 2 \sum_{i=1}^k \frac{m_i}{q_i} \leq 2\sum_{i=1}^k \frac{m_i}{3}\leq \frac{2d}{3}$. This implies that $3d-6< 2d$ and hence $d<6$.  Therefore, $d = 3, 4$ or $5$.

\medskip

\noindent {\it Case 1:} First assume $d = 5$. If $q_1 \geq 4$ then $\frac{m_1}{q_1} + \dots + \frac{m_k}{q_k} \leq \frac{d}{q_1} \leq \frac{5}{4}$. Therefore, by \eqref{eq:2}, $2 > d -2\sum_{i=1}^k\frac{m_i}{q_i} \geq 5 -  \frac{10}{4} = \frac{10}{4}$, a contradiction. So, $q_1 = 3$. If $m_1 \leq 3$ then $3 = d- 2 < 2(\frac{m_1}{q_1} + \dots + \frac{m_k}{q_k}) \leq 2(\frac{m_1}{q_1} + \frac{d-m_1}{q_2}) \leq 2(\frac{m_1}{3} + \frac{5-m_1}{4}) = \frac{15+m_1}{6} \leq \frac{15+3}{6} =3$, a contradiction. So, $m_1 \geq 4$. Since $m_1 \leq d = 5$, it follows that $m_1 = 4$ or $5$.

\smallskip

\noindent {\it 1.1:} Let $m_1 = 5$. Then, $d = m_1$ and $k = 1$. So,  $(q_1^{m_1}, q_2^{m_2}, \dots, q_k^{m_k})
= (3^5)$ and hence $[p_1^{n_1}, \dots, p_{\ell}^{n_{\ell}}] = [3^5]$.
\smallskip

\noindent {\it 1.2:} Let $m_1 = 4$. Then $m_2 = 1$. Therefore, $3  = d-2 < 2\sum_{i=1}^k\frac{m_i}{q_i} = 2(\frac{m_1}{q_1} +  \frac{m_2}{q_2}) = 2(\frac{4}{3} +  \frac{1}{q_2})$. This implies that ${1}/{q_2} > {3}/{2}-{4}/{3}= {1}/{6}$ and hence  $q_2 < 6$. Since $q_2 > q_1 = 3$, $q_2 = 4$ or $5$.

\smallskip

If $q_2=5$, then $(q_1^{m_1}, \dots, q_k^{m_k}) = (3^4, 5^1)$ and hence $[p_1^{n_1}, \dots, p_{\ell}^{n_{\ell}}] = [3^4, 5^1]$.
Similarly, if $q_2 = 4$ then   $[p_1^{n_1}, \dots, p_{\ell}^{n_{\ell}}] = [3^4, 4^1]$.

\medskip

\noindent {\it Case 2:} Now, assume $d = 4$. Then, $1= \frac{d}{2}-1< \sum_{i=1}^k\frac{m_i}{q_i} \leq \frac{d}{q_1} = \frac{4}{q_1}$. So, $q_1 < 4$ and hence $q_1 = 3$.

\smallskip

\noindent {\it 2.1:} If $m_1 = d = 4$, then $[p_1^{n_1}, \dots, p_{\ell}^{n_{\ell}}] = [3^4]$.

\smallskip

\noindent {\it 2.2:} If $m_1 = 3$, then $m_2 = 1$. So, $(q_1^{m_1},  \dots, q_k^{m_k}) = (3^3, q_2^1)$. This implies that  $[p_1^{n_1}, \dots, p_{\ell}^{n_{\ell}}] = [3^3, s^1]$  for some $s \geq 4$. 

\smallskip

\noindent {\it 2.3:} If $m_1 = 2$, then $1  = \frac{d}{2}-1 < \sum_{i=1}^k\frac{m_i}{q_i} =\frac{2}{3}+\frac{m_2}{q_2}+\frac{m_3}{q_3} \leq \frac{2}{3} + \frac{2}{q_2}$. So,  $\frac{2}{q_2} > \frac{1}{3}$ and hence  $q_2< 6$. Thus, $q_2 = 5$ or 4.

\smallskip

\noindent {\it 2.3.1:} If $q_2 = 5$, then $1  = \frac{d}{2}-1  < \frac{2}{3}+\frac{m_2}{5}+\frac{m_3}{q_3}$ and hence $\frac{m_2}{5}+\frac{m_3}{q_3}  > \frac{1}{3}$,   where $m_2+m_3 = d-m_1 =2$ and $m_2\geq 1$. These imply that $q_3\leq 7$.  If $q_3 = 7$ then  $(q_1^{m_1}, \dots, q_k^{m_k}) = (3^2, 5^1, 7^1)$. This implies that $[p_1^{n_1}, \dots, p_{\ell}^{n_{\ell}}] = [3^2, 5^1, 7^1]$ or $[3^1, 5^1, 3^1, 7^1]$. But  $[p_1^{n_1}, \dots, p_{\ell}^{n_{\ell}}] = [3^2, 5^1, 7^1]$ is not possible by Lemma \ref{DM2020} $(i)$ and $[p_1^{n_1}, \dots, p_{\ell}^{n_{\ell}}] = [3^1, 5^1, 3^1, 7^1]$ is not possible by Lemma \ref{DM2020} $(iii)$. So, $q_3 \neq 7$. If $q_3 = 6$, then $(q_1^{m_1}, \dots, q_k^{m_k}) = (3^2, 5^1, 6^1)$. Again, by Lemma \ref{DM2020} $(i)$ and  $(iii)$, $[3^2, 5^1, 6^1]$ and $[3^1, 5^1, 3^1, 6^1]$ are not vertex-types of any maps. So, $q_3 \neq 6$. Since $q_3 > q_2 =5$, it follows that $m_2 = 2$ (and $q_2 = 5$). Then $(q_1^{m_1}, \dots, q_k^{m_k}) = (3^2, 5^2)$. By Lemma \ref{DM2020} $(i)$, $[p_1^{n_1}, \dots, p_{\ell}^{n_{\ell}}] \neq [3^2, 5^2]$. Therefore, $[p_1^{n_1}, \dots, p_{\ell}^{n_{\ell}}] = [3^1, 5^1, 3^1, 5^1]$.

\smallskip

\noindent {\it 2.3.2:} If $q_2 = 4$, then $1= \frac{d}{2}-1< \frac{2}{3}+\frac{m_2}{4}+\frac{m_3}{q_3}$ and hence $\frac{m_2}{4}+\frac{m_3}{q_3} > \frac{1}{3}$. 
If $m_2 = 1$ then $m_3 = 1$. So, $\frac{1}{4}+\frac{1}{q_3} > \frac{1}{q_3}$ and hence $4 < q_3 < 12$. Therefore, $(q_1^{m_1}, \dots, q_k^{m_k}) = (3^2, 4^1, q_3^1)$ and hence $[p_1^{n_1}, \dots, p_{\ell}^{n_{\ell}}] = [3^2, 4^1, q_3^1]$, where $q_3 > 4$. But this is not possible by Lemma \ref{DM2020} $(i)$. So, $m_2 \neq 1$ and hence $m_2 = 2$. Then  $(q_1^{m_1}, \dots, q_k^{m_k}) = (3^2, 4^2)$. By Lemma \ref{DM2020} $(i)$, $[p_1^{n_1}, \dots, p_{\ell}^{n_{\ell}}] \neq [3^2, 4^2]$. Therefore, $[p_1^{n_1}, \dots, p_{\ell}^{n_{\ell}}] = [3^1, 4^1, 3^1, 4^1]$.

\smallskip

\noindent {\it 2.4:} Let $m_1 = 1$. Then, $1 = \frac{d}{2}-1< \frac{1}{3}+\frac{m_2}{q_2}+\frac{m_3}{q_3}+\frac{m_4}{q_4}$, where $m_2+m_3+m_4= 3$ and $4\leq q_2 < q_3<q_4$. These imply that $q_2=4$. If $m_2 = 1$ then $1 < \frac{1}{3} +\frac{1}{4}+\frac{m_3}{q_3}+\frac{m_4}{q_4}  \leq \frac{7}{12} + \frac{2}{q_3}$. So, $\frac{2}{q_3} > \frac{5}{12}$ and hence $q_3 \leq 4=q_2$, a contradiction. Thus, $m_2 \geq 2$  and hence $m_2 = 2$ or $3$.

\smallskip

If $m_2 = 2$, then $m_3 =1$. So,  $1 = \frac{d}{2}-1< \frac{1}{3}+\frac{2}{4}+\frac{1}{q_3}$ and hence $\frac{1}{q_3} > 1-\frac{1}{3}-\frac{1}{2} = \frac{1}{6}$. Therefore, $q_3 < 6$ and hence $q_3 = 5$. Then,  $(q_1^{m_1}, q_2^{m_2}, \dots, q_k^{m_k}) = (3^1, 4^2, 5^1)$. By Lemma \ref{DM2020} $(ii)$,  $[p_1^{n_1}, \dots, p_{\ell}^{n_{\ell}}] \neq [3^1, 4^2, 5^1]$. Therefore, $[p_1^{n_1}, \dots, p_{\ell}^{n_{\ell}}] = [3^1, 4^1, 5^1, 4^1]$.

\smallskip

If $m_2 = 3$, then  $(q_1^{m_1}, q_2^{m_2}, \dots, q_k^{m_k}) = (3^1, 4^3)$ and hence $[p_1^{n_1}, \dots, p_{\ell}^{n_{\ell}}] = [3^1, 4^3]$.

\medskip  

\noindent {\it Case 3:} Finally, assume $d = 3$. Then, $\frac{1}{2}  = \frac{d}{2}-1 < \frac{m_1}{q_1} + \frac{m_2}{q_2} + \frac{m_3}{q_3}$, where $m_1+m_2+m_3=3$ and $3\leq q_2< q_3<q_4$. This implies that $q_1 < 6$ and hence $q_1 = 3, 4$ or $5$.

\smallskip

\noindent {\it 3.1:} Let $q_1 = 5$. Now, $m_1 = 1, 2$ or $3$. If $m_1 = 2$ then $[p_1^{n_1}, \dots, p_{\ell}^{n_{\ell}}] = [5^2, q_2^1]$, where $q_2>5$. This is not possible by Lemma \ref{DM2020} $(i)$. So, $m_1 = 1$ or $3$.

\smallskip

\noindent {\it 3.1.1:} If $m_1 = 1$, then $\frac{1}{2} <\frac{1}{5} + \frac{m_2}{q_2} + \frac{m_3}{q_3}$. So,   $\frac{m_2}{q_2} + \frac{m_3}{q_3} > \frac{1}{2}-\frac{1}{5} = \frac{3}{10}$, where $m_2+m_3=2$ and $5=q_1<q_2<q_3$. These imply that  $q_2 = 6$. If $m_2 = 1$ then $m_3 = 3-m_1-m_2 = 1$ and hence   $[p_1^{n_1}, \dots, p_{\ell}^{n_{\ell}}] = [5^1, 6^1, q_3^1]$, where $q_3 \geq 7$. But, this is not possible by Lemma \ref{DM2020} $(ii)$. Thus, $m_2=2$. Then $[p_1^{n_1}, \dots, p_{\ell}^{n_{\ell}}]=[5^1, 6^2]$.

\smallskip

\noindent {\it 3.1.2:} If $m_1 = 3$, then  $[p_1^{n_1}, \dots, p_{\ell}^{n_{\ell}}] = [5^3]$.

\smallskip

\noindent {\it 3.2:} Let $q_1 = 4$. Since $d=3$, $(m_1, \dots, m_k) = (1, 1, 1), (1, 2), (2, 1)$ or $(3)$.

\smallskip

\noindent {\it 3.2.1:} If $(m_1, \dots, m_k) = (1, 1, 1)$, then $\frac{1}{2} <\frac{1}{4}+\frac{1}{q_2}+\frac{1}{q_3}$. So, $\frac{1}{q_2}+\frac{1}{q_3} > \frac{1}{4}$. Since $q_2 < q_3$, it follows that $q_2 < 8$. If $q_2 = 5$ then  $[p_1^{n_1}, \dots, p_{\ell}^{n_{\ell}}] = [4^1, 5^1, q_3^1]$, $q_3 > 5$. This is not possible by Lemma \ref{DM2020} $(ii)$. So, $q_2 \neq 5$. Similarly, $q_2 \neq 7$. Thus, $q_2 = 6$. Then $\frac{1}{q_3}> \frac{1}{4}-\frac{1}{6} =\frac{1}{12}$ and hence  $q_3<12$. Then, by the same argument,  $q_3 \neq 9, 11$. So, $q_3 = 8$ or $10$. Therefore,  $[p_1^{n_1}, \dots, p_{\ell}^{n_{\ell}}] = [4^1, 6^1, 8^1]$ or $[4^1, 6^1, 10^1]$.

\smallskip

\noindent {\it 3.2.2:} If $(m_1, \dots, m_k) = (1, 2)$, then $\frac{1}{2} <  \frac{1}{4}+\frac{2}{q_2}$ and hence $4 = q_1 < q_2 < 8$. 
Thus, $[p_1^{n_1},$ $\dots, p_{\ell}^{n_{\ell}}] =[4^1, q_2^2]$, $5\leq q_2\leq 7$. 
By Lemma \ref{DM2020} $(i)$, $q_2 \neq 5$ or $7$.  So,  $[p_1^{n_1}, \dots, p_{\ell}^{n_{\ell}}] = [4^1, 6^2]$.

\smallskip

\noindent {\it 3.2.3:} If $(m_1, \dots, m_k) = (2, 1)$, then $[p_1^{n_1}, \dots, p_{\ell}^{n_{\ell}}] = [4^2, q_2^1]$ for some $q_2 \geq 5$.

\smallskip

\noindent {\it 3.2.4:} If $(m_1, \dots, m_k) = (3)$, then  $[p_1^{n_1}, \dots, p_{\ell}^{n_{\ell}}] = [4^3]$.

\smallskip

\noindent {\it 3.3:} Let $q_1 = 3$. By Lemma \ref{DM2020},  $(m_1, \dots, m_k) = (3)$ or $(1, 2)$.

\smallskip

\noindent {\it 3.3.1:} If $(m_1, \dots, m_k) = (3)$, then  $[p_1^{n_1}, \dots, p_{\ell}^{n_{\ell}}] = [q_1^3]= [3^3]$.

\smallskip

\noindent {\it 3.3.2:} If $(m_1, \dots, m_k) = (1, 2)$, then $\frac{1}{2} < \frac{1}{3} + \frac{2}{q_2}$. So,   $q_2 < 12$. Again, by Lemma \ref{DM2020} $(i)$, $q_2$ is not odd. So, $q_2 = 4, 6, 8$ or $10$. Therefore,  $[p_1^{n_1}, \dots, p_{\ell}^{n_{\ell}}] = [3^1, 4^2]$, $[3^1, 6^2]$, $[3^1, 8^2]$ or $[3^1, 10^2]$.
This completes the proof.
\end{proof}

As a consequence we prove 

\begin{corollary} \label{cor6}
Suppose there exists an $n$-vertex  semi-equivelar spherical map of vertex-type  $[p_1^{n_1}, \dots, p_{\ell}^{n_{\ell}}]$.  Then $(n, [p_1^{n_1}, \dots, p_{\ell}^{n_{\ell}}])  = (4, [3^3]), (6, [3^4]), (8, [4^3]), (12, [3^5]), (20, [5^3])$, $(60, [3^4, 5^1]),  (24,[3^4, 4^1]), (30, [3^1, 5^1, 3^1, 5^1]), (12, [3^1, 4^1, 3^1, 4^1])$, $(60, [3^1, 4^1, 5^1, 4^1])$, $(24, [3^1$, $4^3])$, $(60, [5^1$, $6^2]),$ $(48, [4^1, 6^1, 8^1]),$ $(120, [4^1$, $6^1, 10^1]),$ $(24, [4^1, 6^2]),$ $(12, [3^1, 6^2]),$ $(24,$ $ [3^1,$ $ 8^2]),$ $(60, [3^1, 10^2])$, $(6, [3^1, 4^2])$, $(2r, [4^2, r^1])$ for some $r \geq 5$ or $(2s, [3^3, s^1])$ for some $s \geq 4$.
\end{corollary}

\begin{proof}
Let $X$ be an $n$-vertex  semi-equivelar map on $\mathbb{S}^2$ of vertex-type  $[p_1^{n_1}, \dots, p_{\ell}^{n_{\ell}}]$.  Let $f_1, f_2$ be the number of edges and faces of $X$ respectively. Let $d$ be the degree of each vertex. So, $f_1 = (nd)/2$. Consider the $k$-tuple $(q_1^{m_1}, \dots, q_k^{m_k})$, where $3 \le q_1 < \dots <q_k$ and, for each $i=1, \dots,k$, $q_i = p_j$ for some $j$, $m_i = \sum_{p_j = q_i}n_j$. So, $\sum_i m_i = \sum_{j}n_j = d$ and $\sum_i \frac{m_i}{q_i} = \sum_{j}\frac{n_j}{p_j}$. 
A two-way counting the number of ordered pairs $(F, v)$, where $F$ is a $q_i$-gons in $X$ and $v$ is a vertex of $F$, we get: (the number of $q_i$-gons) $\times q_i = n \times m_i$. This implies  that $f_2 = n \times \sum_{i=1}^k\frac{m_i}{q_i} = n\times\sum_{j=1}^{\ell}\frac{n_j}{p_j}$. Since the Euler characteristic of $\mathbb{S}^2$ is $2$, we get 
\begin{align} \label{eq3}
2 & =n-f_1+f_2= n \times (1- \frac{1}{2}\sum_{j=1}^{\ell}n_j +\sum_{j=1}^{\ell}\frac{n_j}{p_j})  = \frac{n}{2} \times (2 - \sum_{j=1}^{\ell}\frac{n_j(p_j-2)}{p_j}). 
\end{align} 
Thus,
\begin{align} \label{eq4}
n = 4\left(2 - \sum_{j=1}^{\ell}\frac{n_j(p_j-2)}{p_j}\right)^{-1}. 
\end{align} 
From \eqref{eq3}, we get  $\sum_{j=1}^{\ell}\frac{n_j(p_j-2)}{p_j} = 2-\frac{4}{n} <2$. 
Therefore, by Theorem \ref{thm:inequality}, $[p_1^{n_1}, \dots, p_{\ell}^{n_{\ell}}]\in {\mathcal A}$. The result now follows from 
\eqref{eq4} and the set ${\mathcal A}$ given in \eqref{a-sum<2}.  (For example, if $[p_1^{n_1}, \dots, p_{\ell}^{n_{\ell}}] = [3^4,5^1]$ then $n = 4(2-(\frac{4(3-2)}{3} + \frac{1(5-2)}{5}))^{-1}= 60$. So, $(n, [p_1^{n_1}, \dots, p_{\ell}^{n_{\ell}}]) = (60, [3^4,5^1])$.) 
\end{proof}


\section{Proofs of main results} \label{sec:proofs-2}

In this sections, we present proofs of Theorems \ref{thm:s2}, \ref{thm:rp2} and Corollaries \ref{cor:s2vt}, \ref{cor:tiling}, \ref{cor:rp2vt}

\begin{lemma}\label{lem3.4}
Let $K$ be a semi-equivelar spherical map. If the number of vertices and the vertex-type of $K$ are same as those of the boundary $\partial P$ of a Platonic solid $P$ then $K \cong \partial P$.
\end{lemma}

\begin{proof}
If $P$ is the tetrahedron, then $K$ is a $4$-vertex triangulation of $\mathbb{S}^2$ and it is trivial to see that $K$ is unique up to isomorphism.

If $P$ is the octahedron, then $K$ is a $6$-vertex triangulation of $\mathbb{S}^2$ and degree of each vertex is $4$. It is easy to see that $K$ is unique up to an isomorphism (cf. \cite{Da1999}). This also implies that the dual map of $K$ which has $8$ vertices and is of vertex-type $[4^3]$ is unique  up to an isomorphism. Hence an 8-vertex map of vertex-type $[4^3]$ on $\mathbb{S}^2$ is isomorphic to the boundary of the cube.

If $P$ is the icosahedron, then $K$ is a $12$-vertex triangulation of $\mathbb{S}^2$ and degree of each vertex is $5$. We know that it is unique up to an isomorphism (cf. \cite[Table 8]{SL2009}, \cite[Lemma 1]{Up2009}). This also implies that the dual map of $K$ which has $20$ vertices and is of vertex-type $[5^3]$ is unique up to an isomorphism.
Hence a 20-vertex map of vertex-type  $[5^3]$ on $\mathbb{S}^2$ is isomorphic to the boundary of the dodecahedron.
 \end{proof}

We need Lemmas \ref{lem3.5}, \ref{lem3.8} and  \ref{lem3.10} to prove Lemma \ref{lem3.12}.

\begin{lemma}\label{lem3.5}
Let $X$ be a semi-equivelar spherical map of vertex-type $[p^1, q^2]$, where $q \geq 6$. If $\alpha, \beta$ are two $p$-gonal faces of $X$ then there exists at most one edge from $\alpha$ to $\beta$.
\end{lemma}

\begin{proof} Since the degree of each vertex is $3$, for each $p$-gonal face $\alpha$ of $X$ and $u \in \alpha$, there exists a unique edge of the form $uv$ where $v \not\in \alpha$ and  $v$ is in another $p$-gon. Consider the graph $G$ whose nodes are $p$-gonal faces of $X$. Two such nodes $\alpha, \beta$ form a link in $G$ if there exists an edge $uv$ in $X$, where $u \in \alpha$ and $v \in \beta$. It is sufficient to show that $G$ is simple.

Suppose $X$ has $n$ vertices. Then, by Corollary \ref{cor6}, $(n, [p^1, q^2]) = (12, [3^1, 6^2]), (24, [3^1, 8^2]),$ $(60,$ $ [3^1, 10^2]), (24, [4^1, 6^2])$ or $(60, [5^1, 6^2])$.  Since $X$ is a polyhedral map, $G$ has no loop. Suppose there is a pair of links between two nodes $\alpha$ and $\beta$ of $G$. Then there exist  $u, v \in \alpha$ and $x, y \in \beta$ such that $e=ux$, $f=vy$ are edges in $X$. Suppose $uv$ is an edge (of $X$). Let $\gamma$ be the $q$-gonal face containing $uv$. Since the degree of each vertex is $3$, it follows that $ux$ $\&$ $vy$ are edges of $\gamma$ and hence $xy$ is a diagonal of $\gamma$. This is not possible since $x, y \in \beta$. So, $uv$ is not an edge. Similarly, $xy$ is a non-edge. Hence $p > 3$.

Let $p =5$. Then $n = 60$ and $q=6$. Since two $p$-gons in $X$ are disjoint, it follows that $G$ has $12$ nodes.  By above $uv, xy$ are non-edges, and there is no more edges between $\alpha$ $ \&$ $\beta$.  Then $\alpha, \beta$ and the edges $e$, $f$ subdivide $\mathbb{S}^2$ into two disks $C, D$ (and interiors of $\alpha$ $\&$  $\beta$). 

\medskip 

\noindent {\bf Claim.} Number of pentagons inside $D$ is at least 6. 

\smallskip

Let ${\mathcal P}$ be the set of pentagons inside $D$. The boundary of $D$ contains $u, v, x, y$ and $m$ (say) vertices, where $2 \leq m \leq 4$. For each of the $m$ vertices, $w$, there exists an edge $wz$ and a pentagon $\gamma$ containing $z$. Then $\gamma$ is inside $D$. So, ${\mathcal P}\neq \emptyset$. Let the number of pentagons in ${\mathcal P}$ be $\ell$. Then,  the numbers of links between $\ell$ nodes is $(5\ell-m)/2$. Since the number of links between any two nodes is at most 2, $\frac{5\ell-m}{2} \leq 2\binom{\ell}{2}$. Thus, $2\ell(\ell-1) \geq 5\ell -m\geq 5\ell -4$ and hence $2\ell^2-7\ell+4\geq 0$. This implies that $\ell\geq 3$. First consider the case where the induced subgraph $G[{\mathcal P}]$ is simple. Then, $\frac{5\ell-m}{2} \leq \binom{\ell}{2}$. Thus, $\ell(\ell-1) \geq 5\ell -m\geq 5\ell -4$ and hence $\ell^2-6\ell+4\geq 0$. This implies that $\ell\geq 6$ and we are done. 

Now, assume that $G[{\mathcal P}]$ is not simple. Suppose there is a pair of links between two nodes $\alpha_1$ and $\beta_1$ of $G[{\mathcal P}]$. Then, as before, there exist edges $e_1=u_1x_1$, $f_1 = v_1y_1$, where $u_1, v_1\in \alpha_1$, $x_1, y_1\in \beta_1$, such that $\alpha_1$, $\beta_1$, $e_1$ and $f_1$  subdivide $\mathbb{S}^{\hspace{.2mm}2}$ into two disks $C_1, D_1$ (and interiors of $\alpha_1$  $\&$ $\beta_1$). Then $\alpha, \beta$ are inside one of $C_1$, $D_1$. Assume, without loss of generality, that $\alpha, \beta$ are inside $C_1$. Then $D_1$ is inside $D$. Let ${\mathcal P}_1$ be the set of pentagons inside $D_1$ and $\ell_1=\#({\mathcal P}_1)$. Again, if $G[{\mathcal P}_1]$ is simple then, by the same arguments as above, $\ell_1\geq 6$. This implies that   the number of pentagons inside $D$ is at least 8 and we are done. So, assume that $G[{\mathcal P}_1]$ is not simple. Then, repeating the same arguments, we get a  pair of links between two nodes $\alpha_2$ $\&$ $\beta_2$ in  $G[{\mathcal P}_1]$ and a disk  $D_2$ bounded by $\alpha_2, \beta_2$ and two edges of the form $e_2=u_2x_2$, $f_2=v_2y_2$, where $u_2, v_2\in \alpha_2$, $x_2, y_2\in\beta_2$, inside $D_1$. Let $\ell_2$ be the number of pentagons inside $D_2$. 
Then, as before, $\ell_2\geq 3$ and hence the number of pentagons inside $D$ is at least $\#(\{\alpha_1, \beta_1, \alpha_2, \beta_2\}) + 3 =7$. This proves the claim.

Similarly, the number of pentagons inside $C$ is at least $6$. Therefore, the number of pentagons in $X$ is $\geq 6+6+2=14$, a contradiction. Thus, $p\neq 5$.  By similar arguments $p \neq 4$. This is a contradiction. So, $G$ is a simple graph. This completes the proof.
\end{proof}

We now present some combinatorial versions of truncation and rectification of polytopes.

\begin{definition}\label{dfn1}
{\rm Let $P$ be a $3$-polytope and $TP$ be the truncation of $P$. Let $X \cong \partial P$ and $V(X) = \{u_1, \dots, u_n\}$. Without loss of generality, we identify $X$ with $\partial P$. Consider a new set (of nodes) $V := \{v_{ij} \colon u_iu_j$ is an edge of $X\}$. So, if $v_{ij} \in V$ then $i \neq j$ and $v_{ji}$ is also in $V$.  Let $E := \{v_{ij}v_{ji} \colon v_{ij} \in V\} \sqcup \{v_{ij}v_{ik} \colon u_j, u_k$  are in a face containing $u_i, 1\le i \le n\}$. Then $(V, E)$ is a graph on $X$. Clearly, from the construction, $(V, E) \cong $ the edge graph of $TP$. Thus, $(V, E)$ gives a map $T(X)$ on $\mathbb{S}^2$. This map $T(X)$ is said to be the} truncation {\rm of $X$}.
\end{definition}

From Definition \ref{dfn1} $\&$ the (geometric) construction of truncation of polytopes we get

\begin{lemma}\label{lem3.8}
Let $X$,  $T(X)$ and $TP$ be as in Definition $\ref{dfn1}$. Then, $T(X)$ is isomorphic to the boundary of $TP$. Moreover, if $X$ is semi-equivelar of vertex-type  $[q^p]$ $($resp., $[p^1, q^1, p^1, q^1])$,  then $T(X)$ is also semi-equivelar and of vertex-type  $[p^1, (2q)^2]$ $($resp., $[4^1, (2p)^1, (2q)^1])$.
\end{lemma}

\begin{proof} Let $P, V(X), V, E$ be as in Definition \ref{dfn1}. Then, from the definition of truncated polytope and the construction in Def. \ref{dfn1}, $T(X) \cong \partial(TP)$.

Let  $X$ be semi-equivelar of $[q^p]$. From the construction in Def. \ref{dfn1}, the set of faces of $T(X)$ is $\{\tilde{\alpha} = v_{i_1i_2}\mbox{-}v_{i_2i_1}\mbox{-}v_{i_2i_3}\mbox{-}v_{i_3i_2}\mbox{-}v_{i_3i_4}\mbox{-}\cdots\mbox{-}v_{i_qi_1}\mbox{-}v_{i_1i_q}\mbox{-}v_{i_1i_2} \colon \alpha = u_{i_1}\mbox{-}u_{i_2}\mbox{-}\cdots\mbox{-}u_{i_q}\mbox{-}u_{i_1}$ is a face of $X\} \sqcup \{\tilde{u_i} = v_{ij_1}\mbox{-}v_{ij_2}\mbox{-}\cdots\mbox{-}v_{ij_q}\mbox{-}v_{ij_1} \colon u_{j_1}\mbox{-}P_1\mbox{-}u_{j_2}\mbox{-}P_2\mbox{-}u_{j_3}\mbox{-}\cdots\mbox{-}P_{q}\mbox{-}u_{j_1}$ is the link\mbox{-}cycle of $u_i\in V(X)\}$.
Thus, the faces incident to the vertex $v_{ij}$ are the $p$-gonal face $\tilde{u_i}$ and the two $2q$-gonal faces $\tilde{\alpha}$ $\&$ $\tilde{\beta}$ where  $\alpha$ $\&$ $\beta$ are faces in $X$ containing the edge $u_iu_j$. Observe that the face-cycle of $v_{ij}$ is $\tilde{u_i}\mbox{-}\tilde{\alpha}\mbox{-}\tilde{\beta}\mbox{-}\tilde{u_i}$. Thus, $T(X)$ is semi-equivelar and the vertex-type  is $[p^1, (2q)^2]$.

Let the vertex-type of $X$ be $[p^1, q^1, p^1, q^1]$. From Def. \ref{dfn1}, the set of faces of $T(X)$ is $\{\tilde{\alpha} := v_{i_1i_2}\mbox{-}v_{i_2i_1}\mbox{-}v_{i_2i_3}\mbox{-}v_{i_3i_2}\mbox{-}v_{i_3i_4}\mbox{-}\cdots\mbox{-}v_{i_ri_1}\mbox{-}v_{i_1i_r}\mbox{-}v_{i_1i_2} \colon \alpha = u_{i_1}\mbox{-}\cdots\mbox{-}u_{i_r}\mbox{-}u_{i_1}$ is a $r$-gonal face of $X, \ r = p, q\}\sqcup  \{\tilde{u_i} := v_{it_1}\mbox{-}v_{it_2} \mbox{-}v_{it_3}\mbox{-}v_{it_4}\mbox{-}v_{it_1} \colon u_{t_1}\mbox{-}P_1\mbox{-}u_{t_2}\mbox{-}P_2 \mbox{-}u_{t_3}\mbox{-}P_{3}\mbox{-}u_{t_4}\mbox{-}P_{4}$ $\mbox{-}u_{t_1}$ is the link\mbox{-}cycle of  $u_i\in V(X)\}$.Thus, the faces incident to the vertex $v_{ij}$ are the square $\tilde{u_i}$, the $2p$-gonal face $\tilde{\alpha}$ and the $2q$-gonal face $\tilde{\beta}$ where $\alpha$ is a $p$-gonal face and $\beta$ is a $q$-gonal face in $X$ containing the edge $u_iu_j$. Observe that the face-cycle of $v_{ij}$ is $\tilde{u_i}\mbox{-}\tilde{\alpha}\mbox{-}\tilde{\beta}\mbox{-}\tilde{u_i}$. Thus, $T(X)$ is semi-equivelar and the vertex-type  is $[4^1, (2p)^1, (2q)^1]$.
\end{proof}

\begin{definition}\label{dfn2}
{\rm Let $P$ be a polytope and $RP$ be the rectification of $P$. Let $X \cong \partial P$ and  $V(X) = \{u_1, \dots, u_n\}$. Without loss of generality, we identify $X$ with $\partial P$. Consider the graph $(V, E)$, where $V$ is the edge set $E(X)$ of $X$ and $E := \{ef \colon e, f$ are two adjacent edges in a face of $X\}$. Then $(V, E)$ is a graph on $X$. From the definition of rectification, it follows that $(V, E) \cong $ the edge graph of $RP$. Thus, $(V, E)$ gives a map say $R(X)$ on $\mathbb{S}^2$, which is said to be the} rectification {\rm of $X$}.
\end{definition}

From Definition \ref{dfn2} $\&$ the (geometric) construction of rectification of polytopes we get

\begin{lemma}\label{lem3.10}
Let $X$, $R(X)$ and $RP$ be as in Definition $\ref{dfn2}$. Then, $R(X)$ is isomorphic to the boundary of $RP$. Moreover, if $X$ is semi-equivelar of vertex-type  $[q^p]$ $($respectively, $[p^1, q^1, p^1, q^1])$,  then $R(X)$ is also semi-equivelar and of vertex-type $[p^1, q^1, p^1, q^1]$ $($respectively, $[4^1, p^1, 4^1, q^1])$.
\end{lemma}


\begin{proof} Let $P, V(X), V, E$ be as in Definition \ref{dfn2}. Then, from the definition of rectified polytope and the  construction in Def. \ref{dfn2}, $R(X) \cong \partial(RP)$.

Let  $X$ be semi-equivelar of vertex-type  $[q^p]$. From the construction in Def. \ref{dfn2}, the set of faces of $R(X)$ is $\{\tilde{\alpha} := e_1 \mbox{-} e_2 \mbox{-} \cdots  \mbox{-} e_q\mbox{-} e_1 \colon \alpha = (e_1 \cap e_2) \mbox{-}\cdots \mbox{-}(e_{q-1}  \cap e_q) \mbox{-} (e_q \cap e_1) \mbox{-}( e_1 \cap e_2)$ is a face of $X\} \sqcup \{\tilde{u_i} := \epsilon_1\mbox{-}\cdots\mbox{-}\epsilon_{p}\mbox{-}\epsilon_{1} \colon u_{j_1}\mbox{-}P_1\mbox{-}u_{j_2}\mbox{-}P_2 \mbox{-} u_{j_3} \mbox{-} \cdots \mbox{-} P_{p} \mbox{-} u_{j_1}$ is the link\mbox{-}cycle of  $u_i\in V(X)$ and $\epsilon_{\ell} = u_iu_{j_{\ell}}, 1 \leq \ell \leq p\}$. The faces incident to the vertex $e=u_iu_j$ are the two squares $\tilde{u_i}$, $\tilde{u_j}$ and the two $q$-gonal faces $\tilde{\alpha}$, $\tilde{\beta}$ where $\alpha$, $\beta$ are faces in $X$ containing the edge $u_iu_j$. Observe that the face-cycle of $e$ ($=u_iu_j = \alpha \cap \beta$) is $\tilde{u_i}\mbox{-}\tilde{\alpha}\mbox{-}\tilde{u_j}\mbox{-}\tilde{\beta}\mbox{-}\tilde{u_i}$. Thus, $R(X)$ is semi-equivelar and the vertex-type is $[p^1, q^1, p^1, q^1]$.

Let $X$ be semi-equivelar of vertex-type   $[p^1, q^1, p^1, q^1]$. From Def. \ref{dfn2}, the set of faces of $R(X)$ is $\{\tilde{\alpha} := e_1\mbox{-}\cdots\mbox{-}e_r\mbox{-}e_1 \colon \alpha = (e_1 \cap e_2) \mbox{-} \mbox{-}\cdots \mbox{-}(e_{r-1} \cap e_r)\mbox{-}(e_r \cap e_1)\mbox{-}(e_1 \cap e_2)$ is a $r$-gonal face of $X, \ r = p, q\}\sqcup \{\tilde{u_i} = e_{1}\mbox{-}e_{2}\mbox{-}e_{3}\mbox{-}e_{4}\mbox{-}e_{1} \colon u_{j_1}\mbox{-}P_1\mbox{-}u_{j_2}\mbox{-}P_2\mbox{-}u_{j_3}\mbox{-}P_{3}\mbox{-}u_{j_4}\mbox{-}P_{4}\mbox{-}u_{j_1}$ is the link-cycle of  $u_i\in  V(X)$ and $e_{\ell} = u_iu_{j_{\ell}}, 1 \leq \ell \leq 4\}$.The faces incident to the vertex $e=u_iu_j$ are the two squares $\tilde{u_i}$ $\&$ $\tilde{u_j}$, the $p$-gonal face $\tilde{\alpha}$ and the $q$-gonal face $\tilde{\beta}$, where $\alpha$ is a $p$-gonal face and $\beta$ is a $q$-gonal face in $X$ containing the edge $u_iu_j$. Observe that the face-cycle of $e$  is $\tilde{u_i}\mbox{-}\tilde{\alpha}\mbox{-}\tilde{u_j}\mbox{-}\tilde{\beta}\mbox{-}\tilde{u_i}$. Thus, $R(X)$ is semi-equivelar and the vertex-type  is $[4^1, p^1, 4^1, q^1]$.
\end{proof}

\begin{lemma}\label{lem3.11}
$(a)$ If $X$ is a $24$-vertex semi-equivelar spherical map of vertex-type  $[3^1, 4^3]$ then $X$ is isomorphic to either the boundary of small rhombicuboctahedron or the boundary of the pseudorhombicuboctahedron. $(b)$ The boundaries of the small rhombicuboctahedron and the pseudorhombicuboctahedron are non-isomorphic.
\end{lemma}

\begin{proof} Since $X$ is a $24$-vertex map of vertex-type  $[3^1, 4^3]$, it follows that $X$ has $8$ triangles and $18$ squares. Since each vertex of $X$ is in one triangle, no two triangles intersect. Thus, a square meets at most two triangles on edges and hence meets at least two other squares on edges. For $2 \le i \le 4$, a square $\alpha$ in $X$ is said to be of type $i$ if $\alpha$ intersects $i$ other squares on edges. Let $s_i$ be the number of squares in $X$ of type $i$. Consider the set $A:= \{(\beta, t) \colon \beta$ $\mbox{ a square of}$ $X,$ $t$ $\mbox{a triangle of}$ $X,$ $\beta \cap t$ $\mbox{is an edge}\}$. Two way counting the cardinality of $A$ gives $2 \times s_2 + 1 \times s_3 + 0 \times s_4 = 8 \times 3$. So, $2s_2+s_3 = 24$. Since $X$ has $18$ squares, $s_2 + s_3 +s_4 = 18$. These give $(s_2, s_3, s_4) = (6, 12, 0), (7, 10, 1), (8, 8, 2), (9, 6, 3), (10, 4, 4), (11, 2, 5)$ or $(12, 0, 6)$.

\smallskip

\noindent {\bf Claim.} There exists a square in $X$ of type $4$.


\begin{figure}[ht]
\tiny
\tikzstyle{ver}=[]
\tikzstyle{vert}=[circle, draw, fill=black!100, inner sep=0pt, minimum width=4pt]
\tikzstyle{vertex}=[circle, draw, fill=black!00, inner sep=0pt, minimum width=4pt]
\tikzstyle{edge} = [draw,thick,-]
\centering

\begin{tikzpicture}[scale=0.8]

\draw ({3.7+2*cos(22.5)}, {2*sin(22.5)}) -- ({3.7+2*cos(67.5)}, {2*sin(67.5)}) -- ({3.7+2*cos(112.5)}, {2*sin(112.5)}) -- ({3.7+2*cos(157.5)}, {2*sin(157.5)}) -- ({3.7+2*cos(202.5)}, {2*sin(202.5)}) -- ({3.5+2*cos(247.5)}, {2*sin(247.5)}) --  ({3+2*cos(292.5)}, {2.2*sin(292.5)}) -- ({4+2*cos(292.5)}, {2*sin(292.5)}) -- ({3.7+2*cos(337.5)}, {2*sin(337.5)}) -- ({3.7+2*cos(22.5)}, {2*sin(22.5)});

\draw ({3+2*cos(292.5)}, {2.2*sin(292.5)}) -- ({2.2+2*cos(292.5)}, {.5*sin(292.5)}) -- ({3+2*cos(292.5)}, {-.7*sin(292.5)}) -- ({3.6+2*cos(292.5)}, {.5*sin(292.5)}) -- ({3+2*cos(292.5)}, {2.2*sin(292.5)});

\draw ({2.2+2*cos(292.5)}, {.5*sin(292.5)}) -- ({3.7+2*cos(157.5)}, {2*sin(157.5)});
\draw ({2.2+2*cos(292.5)}, {.5*sin(292.5)}) -- ({3.5+2*cos(247.5)}, {2*sin(247.5)});
\draw ({3.6+2*cos(292.5)}, {.5*sin(292.5)}) -- ({4+2*cos(292.5)}, {2*sin(292.5)});
\draw ({3.6+2*cos(292.5)}, {.5*sin(292.5)}) -- ({3.7+2*cos(22.5)}, {2*sin(22.5)});
\draw ({3+2*cos(292.5)}, {-.7*sin(292.5)}) -- ({3.7+2*cos(67.5)}, {2*sin(67.5)});
\draw ({3+2*cos(292.5)}, {-.7*sin(292.5)}) -- ({3.7+2*cos(112.5)}, {2*sin(112.5)});

\node[ver] () at ({3+2*cos(292.5)}, {2.5*sin(292.5)}) {$w$};
\node[ver] () at ({3+2*cos(292.5)}, {.5*sin(292.5)}) {$\alpha_2$};
\node[ver] () at ({1.5+2*cos(292.5)}, {.8*sin(292.5)}) {$\alpha_1$};
\node[ver] () at ({3.4+2*cos(157.5)}, {2*sin(157.5)}) {$u_1$};
\node[ver] () at ({4+2*cos(292.5)}, {.5*sin(292.5)}) {$v_2$};
\node[ver] () at ({2.7+2*cos(292.5)}, {-.7*sin(292.5)}) {$a$};
\node[ver] () at ({4+2*cos(22.5)}, {2*sin(22.5)}) {$v_1$};
\node[ver] () at ({1.8+2*cos(292.5)}, {.5*sin(292.5)}) {$u_2$};
\node[ver] () at ({2.4+2*cos(292.5)}, {1.5*sin(292.5)}) {$t_1$};
\node[ver] () at ({3.6+2*cos(292.5)}, {1.5*sin(292.5)}) {$t_2$};
\node[ver] () at ({3.7+2*cos(67.5)}, {2.2*sin(67.5)}) {$c$};
\node[ver] () at ({3.7+2*cos(112.5)}, {2.2*sin(112.5)}) {$b$};
\node[ver] () at ({1.9+2*cos(292.5)}, {1.2+.5*sin(292.5)}) {$\alpha$};
\node[ver] () at ({3.9+2*cos(292.5)}, {1.2+.5*sin(292.5)}) {$\beta$};

\end{tikzpicture}

\vspace{-2mm}
\caption{{Part of the map $X$ when $(s_2,s_3,s_4)=(6,12,0)$.}} \label{fig:4}
\end{figure}

If the claim is not true, then $(s_2, s_3, s_4) = (6, 12, 0)$. Thus, there are $12$ squares of type $3$. Each of these $12$ type $3$ squares meets one triangle on an edge. Since there are $8$ triangles, it follows that there is a triangle $abc$ which meets two type $3$ squares on edges. Assume, with out loss of generality, that $\alpha := abu_1u_2$ and $\beta := acv_1v_2$ are type $3$ squares (see Fig. \ref{fig:4}). 
Since $abu_1u_2$ is type $3$, $u_1u_2$ is in $2$ squares, say in $\alpha~\&~\alpha_1$. Similarly $au_2$ is in two squares. Clearly, they are $\alpha$ $\&$ $\alpha_2 := au_2wv_2$ for some vertex $w$. Then $\alpha, \alpha_1, \alpha_2$  are three squares incident to $u_2$ and hence the $4^{th}$ face incident to $u_2$ is a triangle, say $t_1$, containing $u_2w$. Similarly, there exists a triangle $t_2$ containing  $v_2w$. Since $\alpha_2 = au_2wv_2$ is a face, it follows that $u_2 \neq v_2$ and $u_2v_2$ is a non-edge. These imply  that $t_1 \neq t_2$ are two triangles containing $w$, a contradiction to the fact that each vertex is in  one triangle. This proves the claim.

\begin{figure}[ht]
\tiny
\tikzstyle{ver}=[]
\tikzstyle{vert}=[circle, draw, fill=black!100, inner sep=0pt, minimum width=4pt]
\tikzstyle{vertex}=[circle, draw, fill=black!00, inner sep=0pt, minimum width=4pt]
\tikzstyle{edge} = [draw,thick,-]
\centering

\begin{tikzpicture}[scale=1]

\draw ({3.7+2*cos(22.5)}, {-.5+2*sin(22.5)}) -- ({3.2+2*cos(67.5)}, {2*sin(67.5)}) -- ({4.1+2*cos(112.5)}, {2*sin(112.5)}) -- ({3.7+2*cos(157.5)}, {-.5+2*sin(157.5)}) -- ({3.7+2*cos(202.5)}, {.5+2*sin(202.5)}) -- ({4.1+2*cos(247.5)}, {2*sin(247.5)}) -- ({3.2+2*cos(292.5)}, {2*sin(292.5)}) -- ({3.7+2*cos(337.5)}, {.5+2*sin(337.5)}) -- ({3.7+2*cos(22.5)}, {-.5+2*sin(22.5)});

\draw ({3.2+2*cos(67.5)}, {2*sin(67.5)}) -- ({3.2+2*cos(292.5)}, {2*sin(292.5)});
\draw ({4.1+2*cos(247.5)}, {2*sin(247.5)}) -- ({4.1+2*cos(112.5)}, {2*sin(112.5)});

\draw ({3.7+2*cos(22.5)}, {-.5+2*sin(22.5)}) -- ({3.7+2*cos(157.5)}, {-.5+2*sin(157.5)});
\draw ({3.7+2*cos(337.5)}, {.5+2*sin(337.5)}) -- ({3.7+2*cos(202.5)}, {.5+2*sin(202.5)});

\draw ({3+2*cos(22.5)}, {-.5+2*sin(22.5)}) -- ({3.2+2*cos(67.5)}, {-.7+2*sin(67.5)}) -- ({4.1+2*cos(112.5)}, {-.7+2*sin(112.5)}) -- ({.6+2*cos(22.5)}, {-.5+2*sin(22.5)}) --
({.6+2*cos(22.5)}, {-1.02+2*sin(22.5)}) -- ({4.1+2*cos(112.5)}, {-3+2*sin(112.5)}) -- ({4.73+2*cos(112.5)}, {-3+2*sin(112.5)}) -- ({3+2*cos(22.5)}, {-1.02+2*sin(22.5)}) -- ({3+2*cos(22.5)}, {-.5+2*sin(22.5)});

\node[ver] () at ({4+2*cos(22.5)}, {-.5+2*sin(22.5)}) {$u_{15}$};
\node[ver] () at ({3.2+2*cos(67.5)}, {.2+2*sin(67.5)}) {$u_{14}$};
\node[ver] () at ({4.1+2*cos(112.5)}, {.2+2*sin(112.5)}) {$u_{13}$};
\node[ver] () at ({3.4+2*cos(157.5)}, {-.5+2*sin(157.5)}) {$u_{20}$};
\node[ver] () at ({3.4+2*cos(202.5)}, {.5+2*sin(202.5)}) {$u_{19}$};
\node[ver] () at ({4.1+2*cos(247.5)}, {-.2+2*sin(247.5)}) {$u_{18}$};
\node[ver] () at ({3.2+2*cos(292.5)}, {-.2+2*sin(292.5)}) {$u_{17}$};
\node[ver] () at ({4+2*cos(337.5)}, {.5+2*sin(337.5)}) {$u_{16}$};

\node[ver] () at ({3.1+2*cos(22.5)}, {-.37+2*sin(22.5)}) {$u_{7}$};
\node[ver] () at ({3.4+2*cos(67.5)}, {-.7+2*sin(67.5)}) {$u_{6}$};
\node[ver] () at ({3.9+2*cos(112.5)}, {-.7+2*sin(112.5)}) {$u_{5}$};
\node[ver] () at ({.5+2*cos(22.5)}, {-.37+2*sin(22.5)}) {$u_{12}$};
\node[ver] () at ({.5+2*cos(22.5)}, {-1.23+2*sin(22.5)}) {$u_{11}$};
\node[ver] () at ({3.8+2*cos(112.5)}, {-3+2*sin(112.5)}) {$u_{10}$};
\node[ver] () at ({5+2*cos(112.5)}, {-3+2*sin(112.5)}) {$u_{9}$};
\node[ver] () at ({3.1+2*cos(22.5)}, {-1.25+2*sin(22.5)}) {$u_{8}$};

\node[ver] () at ({1.8+2*cos(22.5)}, {-.77+2*sin(22.5)}) {$\alpha$};

\node[ver] () at ({2.3+2*cos(22.5)}, {-.37+2*sin(22.5)}) {$u_{2}$};
\node[ver] () at ({1.3+2*cos(22.5)}, {-.37+2*sin(22.5)}) {$u_{1}$};
\node[ver] () at ({2.3+2*cos(22.5)}, {-1.2+2*sin(22.5)}) {$u_{3}$};
\node[ver] () at ({1.3+2*cos(22.5)}, {-1.2+2*sin(22.5)}) {$u_{4}$};

\node[ver] () at ({4.3+2*cos(292.5)}, {2*sin(292.5)}) {$A=D_1$};


\end{tikzpicture}\hspace{.2cm}
\begin{tikzpicture}[scale=0.8]

\draw  ({3.7+2*cos(22.5)}, {2*sin(22.5)}) -- ({3.7+2*cos(67.5)}, {2*sin(67.5)}) -- ({3.7+2*cos(112.5)}, {2*sin(112.5)}) -- ({3.7+2*cos(157.5)}, {2*sin(157.5)}) -- ({3.7+2*cos(202.5)}, {2*sin(202.5)}) -- ({3.7+2*cos(247.5)}, {2*sin(247.5)}) -- ({3.7+2*cos(292.5)}, {2*sin(292.5)}) -- ({3.7+2*cos(337.5)}, {2*sin(337.5)}) -- ({3.7+2*cos(22.5)}, {2*sin(22.5)});

\draw ({3.7+2*cos(22.5)}, {2*sin(22.5)}) -- ({3.7+2*cos(157.5)}, {2*sin(157.5)});
\draw ({3.7+2*cos(337.5)}, {2*sin(337.5)}) -- ({3.7+2*cos(202.5)}, {2*sin(202.5)});

\draw ({3.7+2*cos(67.5)}, {2*sin(67.5)}) -- ({3.7+2*cos(292.5)}, {2*sin(292.5)});
\draw ({3.7+2*cos(247.5)}, {2*sin(247.5)}) -- ({3.7+2*cos(112.5)}, {2*sin(112.5)});

\node[ver] () at ({4+2*cos(22.5)}, {.2+2*sin(22.5)}) {$u_{16}$};
\node[ver] () at ({3.8+2*cos(67.5)}, {.2+2*sin(67.5)}) {$u_{17}$};
\node[ver] () at ({3.8+2*cos(112.5)}, {.2+2*sin(112.5)}) {$u_{18}$};
\node[ver] () at ({3.4+2*cos(157.5)}, {.3+2*sin(157.5)}) {$u_{19}$};
\node[ver] () at ({3.4+2*cos(202.5)}, {-.3+2*sin(202.5)}) {$u_{20}$};
\node[ver] () at ({3.9+2*cos(247.5)}, {-.3+2*sin(247.5)}) {$u_{13}$};
\node[ver] () at ({3.9+2*cos(292.5)}, {-.3+2*sin(292.5)}) {$u_{14}$};
\node[ver] () at ({4+2*cos(337.5)}, {-.3+2*sin(337.5)}) {$u_{15}$};

\node[ver] () at ({3.3+2*cos(112.5)}, {-1.3+2*sin(112.5)}) {$u_{24}$};
\node[ver] () at ({5.6+2*cos(112.5)}, {-1.3+2*sin(112.5)}) {$u_{23}$};

\node[ver] () at ({3.3+2*cos(112.5)}, {-2.4+2*sin(112.5)}) {$u_{21}$};
\node[ver] () at ({5.6+2*cos(112.5)}, {-2.4+2*sin(112.5)}) {$u_{22}$};

\node[ver] () at ({4.7+2*cos(112.5)}, {-4.5+2*sin(112.5)}) {$B$};
\node[ver] () at ({4.5+2*cos(112.5)}, {-1.8+2*sin(112.5)}) {$\beta$};

\end{tikzpicture}\hspace{.2cm}
\begin{tikzpicture}[scale=0.8]

\draw  ({3.7+2*cos(22.5)}, {2*sin(22.5)}) -- ({3.7+2*cos(67.5)}, {2*sin(67.5)}) -- ({3.7+2*cos(112.5)}, {2*sin(112.5)}) -- ({3.7+2*cos(157.5)}, {2*sin(157.5)}) -- ({3.7+2*cos(202.5)}, {2*sin(202.5)}) -- ({3.7+2*cos(247.5)}, {2*sin(247.5)}) -- ({3.7+2*cos(292.5)}, {2*sin(292.5)}) -- ({3.7+2*cos(337.5)}, {2*sin(337.5)}) -- ({3.7+2*cos(22.5)}, {2*sin(22.5)});

\draw ({3.7+2*cos(22.5)}, {2*sin(22.5)}) -- ({3.7+2*cos(157.5)}, {2*sin(157.5)});
\draw ({3.7+2*cos(337.5)}, {2*sin(337.5)}) -- ({3.7+2*cos(202.5)}, {2*sin(202.5)});

\draw ({3.7+2*cos(67.5)}, {2*sin(67.5)}) -- ({3.7+2*cos(292.5)}, {2*sin(292.5)});
\draw ({3.7+2*cos(247.5)}, {2*sin(247.5)}) -- ({3.7+2*cos(112.5)}, {2*sin(112.5)});

\node[ver] () at ({4+2*cos(22.5)}, {.2+2*sin(22.5)}) {$u_{17}$};
\node[ver] () at ({3.8+2*cos(67.5)}, {.2+2*sin(67.5)}) {$u_{18}$};
\node[ver] () at ({3.8+2*cos(112.5)}, {.2+2*sin(112.5)}) {$u_{19}$};
\node[ver] () at ({3.4+2*cos(157.5)}, {.3+2*sin(157.5)}) {$u_{20}$};
\node[ver] () at ({3.4+2*cos(202.5)}, {-.3+2*sin(202.5)}) {$u_{13}$};
\node[ver] () at ({3.9+2*cos(247.5)}, {-.3+2*sin(247.5)}) {$u_{14}$};
\node[ver] () at ({3.9+2*cos(292.5)}, {-.3+2*sin(292.5)}) {$u_{15}$};
\node[ver] () at ({4+2*cos(337.5)}, {-.3+2*sin(337.5)}) {$u_{16}$};

\node[ver] () at ({3.3+2*cos(112.5)}, {-1.3+2*sin(112.5)}) {$u_{24}$};
\node[ver] () at ({5.6+2*cos(112.5)}, {-1.3+2*sin(112.5)}) {$u_{23}$};

\node[ver] () at ({3.3+2*cos(112.5)}, {-2.4+2*sin(112.5)}) {$u_{21}$};
\node[ver] () at ({5.6+2*cos(112.5)}, {-2.4+2*sin(112.5)}) {$u_{22}$};

\node[ver] () at ({4.7+2*cos(112.5)}, {-4.5+2*sin(112.5)}) {$C$};

\node[ver] () at ({4.5+2*cos(112.5)}, {-1.8+2*sin(112.5)}) {$\beta$};

\end{tikzpicture}

\vspace{-1mm}
\caption{{$A$ is a part of $X$ with 13 squares. $B$ and $C$ are complements of $A$ in $X$.}} \label{fig:5}
\end{figure}

Let $\alpha = u_1u_2u_3u_4$ be a type $4$ square. Let the four squares which intersect $\alpha$ on edges be $u_1u_2u_6u_5, u_2u_3u_8u_7, u_3u_4u_{10}u_9~\&~ u_1u_4u_{11}u_{12}$ (see Fig. \ref{fig:5}). Then $u_2u_6u_7, u_3u_8u_9, u_4u_{10}u_{11}$ and $u_1u_5u_{12}$ must be triangles. If $u_i = u_j$ for $5 \le i \neq j \le 12$ then $u_i$ is in two triangles, a contradiction. Thus, $u_1, \dots, u_{12}$ are distinct vertices. For $5 \le i \le 12$, $u_i$ is in two more squares. This gives us $8$ squares of the form $F_1=u_5u_6u_{14}u_{13}$, $F_2=u_6u_7u_{15}u_{14}$, $F_3=u_7u_8u_{16}u_{15}$, $F_4=u_8u_9u_{17}u_{16}$, $F_5=u_9u_{10}u_{18}u_{17}$, $F_6=u_{10}u_{11}u_{19}u_{18}$, $F_7=u_{11}u_{12}u_{20}u_{19}$, $F_8=u_{12}u_{5}u_{13}u_{20}$.  Suppose $F_i=F_j$ for some $i\neq  j$. Assume, without loss of generality, that $F_1 = F_j$ for some $j\neq 1$.  Clearly, $j \neq  2, 3, 7, 8$.  If $F_1=F_4$ then, since $u_5, \dots, u_{12}$ are  distinct,  $u_5= u_{16}$ or $u_{17}$. In either case, the number of edges containing $u_5$ would be more than 4, a contradiction. Thus, $j\neq 4$. By similar arguments, $j\neq 5, 6$. Thus, $F_1, \dots, F_8$ are distinct.

If $u_k = u_{\ell}$ for $13 \le k \neq \ell \le 20$, then $u_k$ would be in $4$ squares, a contradiction. This  implies that $u_1, \dots, u_{20}$ are distinct vertices and we have a disk $D_1$ with these $20$ vertices, $13$ squares and $4$ triangles. Then the remaining $5$ squares and $4$ triangles give a disk $D_2$ such that $\partial D_1 = \partial D_2$. Since each vertex on $\partial D_1$ is in $2$ squares of $D_1$, it follows that each vertex on $\partial D_2$ is in one triangle and one square of $D_2$. Therefore, the $4$ remaining vertices of $X$ are inside $D_2$ and form a square, say $\beta := u_{21}u_{22}u_{23}u_{24}$.
Then $4$ remaining squares in $D_2$ have one edge common with $\partial D_2$ and one each with $\beta$. Since $\partial D_1 = \partial D_2$ and each internal vertex is in $3$ squares, it follows that, up to isomorphism, $D_2$ is $B$ or $C$. If $D_2 = B$ then $X$ is isomorphic to small rhombicuboctahedron and if $D_2 = C$ then $X$ is isomorphic to the boundary of the pseudorhombicuboctahedron. This proves part $(a)$.

Observe that $(s_2, s_3, s_4) = (12, 0, 6)$ for the small rhombicuboctahedron and $(s_2, s_3, s_4)= (8, 8, 2)$  for the pseudorhombicuboctahedron. Thus, their boundaries are non-isomorphic. This proves part $(b)$.  
\end{proof} 

\begin{lemma}\label{lem3.12}
Let $K$ be a semi-equivelar spherical map. Suppose the vertex-type of $K$ is not $[3^1, 4^3]$, $[3^4, 4^1]$ or $[3^4, 5^1]$. If the number of vertices and vertex-type of $K$ are same as those of the boundary $\partial P$ of an Archimedean solid $P$ then $K \cong \partial P$.
\end{lemma}

\begin{proof}
Let  $K$ be an $n$-vertex  map of  vertex-type $[p_1^{n_1}, \dots, p_{\ell}^{n_{\ell}}]$. Then, $(n, [p_1^{n_1}, \dots,$ $p_{\ell}^{n_{\ell}}])$ = $(12, [3^1, 6^2]),$ $(12, [3^1, 4^1, 3^1, 4^1]),$ $(24, [3^1, 8^2]),$ $(30, [3^1, 5^1,$ $3^1, 5^1]),$ $(60, [3^1, 4^1, 5^1, 4^1]),$ $(60,$ $[5^1, 6^2]),$ $(48, [4^1, 6^1, 8^1]), (120, [4^1, 6^1, 10^1]),$ $(24,$ $[4^1, 6^2])$ or $(60, [3^1, 10^2])$.

If $(n, [p_1^{n_1}, \dots, p_{\ell}^{n_{\ell}}]) = (60, [5^1, 6^2])$, then $K$ has twelve pentagons which partition the vertex set. Consider the graph $G$ whose nodes are the pentagons of $K$. Two nodes are joined by a link in $G$ if the corresponding faces are joined by an edge of $K$. (Observe that this $G$ is same as in the proof of Lemma \ref{lem3.5}.) By the proof of Lemma \ref{lem3.5}, $G$ is a simple $5$-regular graph and can be drawn on $\mathbb{S}^2$. This gives a map $[K]$ whose faces are triangles. (Each hexagon of $K$ gives a 3-cycle.) Since $[K]$ has $12$ vertices, by Lemma \ref{lem3.4}, $[K]$ is the boundary of the icosahedron. By Lemma \ref{lem3.8}, $T([K])$ is the boundary of the truncated icosahedron. Clearly, the operations $X \mapsto [X] $ and $Y \mapsto T(Y)$ are inverse operations. Thus, $K$ is isomorphic to $T([K])$. Therefore, $K$ is isomorphic to the boundary of the truncated icosahedron. Similarly, by Lemmas \ref{lem3.4}, \ref{lem3.5} $\&$ \ref{lem3.8}, the maps of vertex-types $[3^1, 6^2], [3^1, 8^2], [3^1, 10^2]$ and $[4^1, 6^2]$ are isomorphic to the boundaries of the truncated tetrahedron, the truncated cube, the truncated dodecahedron and  the truncated octahedron respectively.

If $(n, [p_1^{n_1}, \dots, p_{\ell}^{n_{\ell}}]) = (12, [3^1, 4^1, 3^1, 4^1])$, then $K$ consists of $8$ triangles and $6$ squares. Consider the graph $G_1$ whose nodes are the triangles of $K$. Two nodes $z_i, z_j$ are joined by a link $z_iz_j$ in $G_1$ if $z_i \cap z_j \neq \emptyset$. Since the intersecting triangles in $K$ meet on a vertex and through each vertex there are exactly two triangles, it follows that $G_1$ is a simple $3$-regular graph. Since $K$ is a map on $\mathbb{S}^2$, $G_1$ can be drawn on $\mathbb{S}^2$. Thus, $G_1$  gives a map $\overline{K}$ on $\mathbb{S}^2$ with $8$ vertices. The four triangles intersecting a square of $K$ form a face of $\overline{K}$. So, $\overline{K}$ is of vertex-type $[4^3]$.
Thus, $\overline{K}$ is the boundary of the cube. Hence, by Lemma \ref{lem3.10} and Proposition \ref{prop:cox}, $R(\overline{K})$ is the boundary of cuboctahedron. Clearly, the operations $X \mapsto \overline{X} ~\&~ Y \mapsto R(Y)$ are inverse operations. Thus, $K$ is isomorphic to $R(\overline{K})$. So, $K$ is isomorphic to the boundary of cuboctahedron. Similarly, a $30$-vertex map of vertex-type $[3^1, 5^1, 3^1, 5^1]$ is isomorphic to the boundary of icosidodecahedron.

If $(n, [p_1^{n_1}, \dots, p_{\ell}^{n_{\ell}}]) = (60, [3^1, 4^1, 5^1, 4^1])$, then $K$ consists of $20$ triangles, $30$ squares and $12$ pentagons. Consider the graph $G_2$ whose nodes are the squares of $K$. Two nodes $x_i, x_j$ are joined by a link $x_ix_j$ in $G_2$ if $x_i \cap x_j \neq \emptyset$. Since the intersecting squares in $K$ meet on a vertex and through each vertex there are exactly two squares, it follows that $G_2$ is a simple $4$-regular graph. Again,  $G_2$ can be drawn on $\mathbb{S}^2$, and hence $G_2$  gives a $30$-vertex map $\overline{K}$ on $\mathbb{S}^2$. The three squares intersecting a triangle and the four squares intersecting a pentagon form faces of $\overline{K}$. So, $\overline{K}$ is of vertex-type $[3^1, 5^1, 3 ^1, 5^1]$. Thus, by the observation in the previous paragraph, $\overline{K}$ is the boundary of the icosidodecahedron. Hence, by Lemma \ref{lem3.10} and Proposition \ref{prop:cox}, $R(\overline{K})$ is the boundary of small rhombicosidodecahedron. Clearly, the operations $X \mapsto \overline{X} ~\&~ Y \mapsto R(Y)$ are inverse operations. Thus, $K$ is isomorphic to $R(\overline{K})$. Therefore, $K$ is isomorphic to the boundary of small rhombicosidodecahedron.

Finally, if $(n, [p_1^{n_1}, \dots, p_{\ell}^{n_{\ell}}]) = (120, [4^1, 6^1, 10^1])$, then $K$ consists of $30$ squares, $20$ hexagons and $12$ decagons. Consider the graph $G_3$ whose nodes are the squares of $K$. Two nodes $y_i, y_j$ are joined by a link $y_iy_j$ in $G_3$ if there exists a hexagon $F$ in $K$ such that $y_i \cap F \neq \emptyset$ $\&$ $y_j \cap F \neq \emptyset$.  Since $K$ is a polyhedral map, the adjacent squares of each hexagon in $K$ are distinct. This implies that $G_3$ is a simple $4$-regular graph and can be drawn on $\mathbb{S}^2$. This gives a $30$-vertex map $\overline{K}$ on $\mathbb{S}^2$. The three squares intersecting a hexagon and the five squares intersecting a decagon  form faces of $\overline{K}$. So, $\overline{K}$ is of vertex-type $[3^1, 5^1, 3 ^1, 5^1]$. Thus, (as above)  $\overline{K}$ is the boundary of the icosidodecahedron. Hence by Lemma \ref{lem3.8} and Proposition \ref{prop:cox}, $R(\overline{K})$ is the boundary of great rhombicosidodecahedron. Clearly, the operations $X \mapsto [X]$ $ \&$ $Y \mapsto T(Y)$ are inverse operations. Thus, $K$ is isomorphic to $T([K])$. Therefore, $K$ is isomorphic to the boundary of great rhombicosidodecahedron.
Similarly, a $48$-vertex map of vertex-type $[4^1, 6^1, 8^1]$ is isomorphic to the boundary of great rhombicuboctahedron.
\end{proof}

\begin{lemma} \label{lem3.13}
A semi-equivelar spherical map of vertex-type $[3^4, 4^1]$  with $24$ vertices (respectively, $[3^4, 5^1]$ with $60$ vertices) is unique.
\end{lemma}

\begin{proof}
Let $X$ be a $24$-vertex map of vertex-type $[3^4, 4^1]$. So, $X$ has $\frac{24\times 4}{3}=32$ triangles and $\frac{24}{6}=6$ squares. Let us call an edge blue if it is the intersection of  two triangles and red if it is the intersection of one triangle and one square. So, each vertex is in two red edges and (hence) in  three blue edges. 
We say a blue  edge $uv$  is {\em deep blue} if 
 three triangles containing $u$ (respectively, $v$) lie on one side of $uv$ and one lies on the other side of $uv$. (Clearly, if $uvx$ and $uvy$ are triangles then  $uv$ is deep blue if and only if one of $ux$, $uy$ (respectively, $vx$, $vy$) is red. Since  two red edges are disjoint, it follows that an edge $uv$ is deep blue  if and only if there exist vertices $x$, $y$ such that $ux$, $vy$ are red and $uy$, $vx$ are blue.) 

\medskip 

\noindent {\bf Claim.} For  each vertex $u$, there exists a unique deep blue edge through $u$. 

\smallskip

Let the faces containing $u$ be $uv_1v_2$, $uv_2v_3$, $uv_3v_4$, $uv_4v_5$, $uv_5w_1v_1$.  So, $uv_3$ is not  deep blue.  Since $v_1u$ and $v_1w_1$ are red edges, $v_1v_2$ is blue. Let  $v_1v_2w_2$ be a triangle containing $v_1v_2$.  Since $v_3$ is in a unique square, 
$v_2v_3$ and $v_3v_4$ cannot be simultaneously red. Suppose both are blue. Let the face containing $v_2v_3$ (respectively, $v_3v_4$) and not containing $u$ be $v_2v_3w_4$ (respectively, $v_3v_4w_5$) (see Fig. \ref{fig:6} (a)). Then the square containing $v_2$ must contain $v_2w_4$ and the square containing $v_3$ must contain $v_3w_4$. Thus, $w_4$ is in two squares. This is not possible since vertex-type is $[3^4, 4^1]$. So, exactly one of $v_2v_3$,  $v_3v_4$ is red. If  $v_3v_4$ is red and $v_2v_3$ is blue (see Fig. \ref{fig:6} (b)) then  $uv_4$ is deep blue and $uv_2$ is not deep blue. In the other case,  $uv_2$ is deep blue and $uv_4$ is not deep blue. This proves the claim.


\smallskip

It follows by the claim that  the deep blue edges form a perfect matching in the edge graph of $X$. 
Let $\widetilde{X}$ be the map obtained from $X$ by removing all the $12$ deep blue edges. (This makes $24$ triangles to $12$ squares.) Then $\widetilde{X}$ is a $24$-vertex map with $32-24=8$ triangles and $6+12=18$ squares. Moreover, the degree of each vertex in $\widetilde{X}$ is $5-1=4$. 
Since $yw$ and $wz$ cannot both be deep blue (see Fig. \ref{fig:6} (c)). Thus,  one of $\gamma$, $\delta$ is a triangle in $\widetilde{X}$. So, $u$ is in at least one triangle in $\widetilde{X}$. Since there are 8 triangles and 24 vertices, it follows that each vertex is in exactly one triangle and hence in $4-1=3$ squares. Thus, $\widetilde{X}$ is a semi-equivelar map of  vertex-type $[3^1, 4^3]$.

Let $xuyv$ is a new square in $\widetilde{X}$, where $uv$ is a deep blue edge in $X$. Assume that $xuv$ is the unique triangle in $X$ containing $u$ on one side of $uv$. Then $ux$ is a red  edge and is in a square $\alpha$ (see Fig. \ref{fig:6} (c)). This implies that the other three faces (other than $xuv$ $ \&$  $\alpha$) incident to $x$ are triangles in $X$. Thus, $vx$ is blue and hence $vx$ is in a triangle $\beta$ of $\widetilde{X}$. Similarly $uy$ is in a triangle $\gamma$ in $\widetilde{X}$. These imply that the new square $xuyv$ in $\widetilde{X}$ is of type $2$ (as in the proof of Lemma \ref{lem3.11}). Therefore, $\widetilde{X}$ has at least $12$ type $2$ squares and hence by the proof of Lemma \ref{lem3.11}, $\widetilde{X}$  is the boundary of small rhombicuboctahedron.

\begin{figure}[ht]

\tikzstyle{ver}=[]
\tikzstyle{vert}=[circle, draw, fill=black!100, inner sep=0pt, minimum width=4pt]
\tikzstyle{vertex}=[circle, draw, fill=black!00, inner sep=0pt, minimum width=4pt]
\tikzstyle{edge} = [draw,thick,-]
\centering

\begin{tikzpicture}[scale=0.7]

\draw ({sqrt(3)}, 1) -- (0, 2) -- ({-sqrt(3)}, 1) -- ({-sqrt(3)}, -1) -- (0, -2) -- ({sqrt(3)}, -1) -- ({sqrt(3)}, 1);

\draw ({-sqrt(3)}, -1) -- (-3, 0) -- ({-sqrt(3)}, 1); \draw ({sqrt(3)}, -1) -- (3, 0) -- ({sqrt(3)}, 1);

\draw[edge, dotted] ({-sqrt(3)}, 1) -- ({-sqrt(3)}, 3) -- (0, 2); \draw[edge, dotted] ({sqrt(3)}, 1) -- ({sqrt(3)}, 3) -- (0, 2);

\draw (0, 0) -- ({-sqrt(3)}, 1); \draw (0, 0) -- ({-sqrt(3)}, -1);
\draw (0, 0) -- (0, 2); \draw (0, 0) -- ({sqrt(3)}, 1); \draw (0, 0) -- ({sqrt(3)}, -1);

\node[ver] () at (0, -.3) {$u$}; 
\node[ver] () at ({sqrt(3)}, -1.5) {$v_5$}; 
\node[ver] () at ({2.1}, 1.3) {$v_4$};
\node[ver] () at (0, 2.4) {$v_3$};
\node[ver] () at ({-2.1}, 1.3) {$v_2$};
\node[ver] () at ({-sqrt(3)}, -1.5) {$v_1$}; 
\node[ver] () at (0, -1.6) {$w_1$}; 
\node[ver] () at (-3.2, .4) {$w_2$}; 
\node[ver] () at (3.2, -.5) {$w_3$}; 
\node[ver] () at ({-2.1}, 2.8) {$w_4$}; 
\node[ver] () at ({2.2}, 2.8) {$w_5$}; 

\node[ver] () at (-2.5, -2) {$(a)$};  

\end{tikzpicture}
\begin{tikzpicture}[scale=0.7]

\draw ({sqrt(3)}, 1) -- (0, 2) -- ({-sqrt(3)}, 1) -- ({-sqrt(3)}, -1) -- (0, -2) -- ({sqrt(3)}, -1) -- ({sqrt(3)}, 1);

\draw ({-sqrt(3)}, -1) -- (-3, 0) -- ({-sqrt(3)}, 1); \draw ({sqrt(3)}, -1) -- (3, 0) -- ({sqrt(3)}, 1);

\draw ({-sqrt(3)}, 1) -- ({-sqrt(3)}, 3) -- (0, 2); 


\draw ({sqrt(3)}, 1) -- (3, 2) -- ({sqrt(3)}, 3) -- (0, 2);

\draw (0, 0) -- ({-sqrt(3)}, 1); \draw (0, 0) -- ({-sqrt(3)}, -1);
\draw (0, 0) -- (0, 2); \draw (0, 0) -- ({sqrt(3)}, 1); \draw (0, 0) -- ({sqrt(3)}, -1);

\node[ver] () at (0, -.3) {$u$}; 
\node[ver] () at ({sqrt(3)}, -1.5) {$v_5$}; 
\node[ver] () at (1.7, 1.5) {$v_4$};
\node[ver] () at (0, 2.4) {$v_3$};
\node[ver] () at ({-2.1}, 1.3) {$v_2$};
\node[ver] () at ({-sqrt(3)}, -1.5) {$v_1$}; 
\node[ver] () at (0, -1.6) {$w_1$}; 
\node[ver] () at (-3.2, .4) {$w_2$}; 
\node[ver] () at (3.2, -.5) {$w_3$}; 
\node[ver] () at ({-2.2}, 2.8) {$w_4$}; 
\node[ver] () at ({2.3}, 3) {$w_5$}; 
\node[ver] () at (3.1, 1.5) {$w_6$}; 

\node[ver] () at (-2.5, -2) {$(b)$};  


\end{tikzpicture}
\begin{tikzpicture}[scale=0.7]

\draw ({sqrt(3)}, 1) -- (0, 2) -- ({-sqrt(3)}, 1) -- ({-sqrt(3)}, -1) -- (0, -2) -- ({sqrt(3)}, -1) -- ({sqrt(3)}, 1);

\draw ({-sqrt(3)}, -1) -- (-3, 0) -- ({-sqrt(3)}, 1); 

\draw[edge, dashed]  (0, 0) -- ({-sqrt(3)}, 1); \draw (0, 0) -- ({-sqrt(3)}, -1); 

\draw (0, 0) -- (0, 2); \draw (0, 0) -- ({sqrt(3)}, 1); \draw (0, 0) -- ({sqrt(3)}, -1);

\node[ver] () at (0, -.3) {$u$};
\node[ver] () at (0.4, 2.1) {$y$};
\node[ver] () at ({-sqrt(3)}, 1.3) {$v$};
\node[ver] () at ({-sqrt(3)}, -1.5) {$x$}; 
\node[ver] () at ({sqrt(3)}, 1.3) {$w$};
\node[ver] () at ({sqrt(3)}, -1.5) {$z$};
\node[ver] () at (0, -1) {$\alpha$};
\node[ver] () at ({2.3-sqrt(3)}, 1) {$\gamma$};
\node[ver] () at (-2.3, 0) {$\beta$}; 
\node[ver] () at (1.3, 0) {$\delta$}; 

\node[ver] () at (-2.5, -2) {$(c)$}; 

\end{tikzpicture}
\caption{{(a) and (b): Parts of $X$ as in the proof of above claim. (c): Part of $\widetilde{X}$.}} \label{fig:6}
\end{figure}

 On the other hand, if $Y$ is the boundary of small rhombicuboctahedron then, as in the proof of Lemma \ref{lem3.11}, $Y$ has  twelve  type $2$ squares. We can divide each of these type $2$ squares into two triangles by adding one diagonal in two ways. It is not difficult to see that if we chose one diagonal in one type $2$ squares then there exists exactly one diagonal for each of the other 11 type $2$ squares so that these $12$ new diagonals form a perfect matching. 
 And we get a map $\widehat{Y}$ of vertex-type $[3^4, 4^1]$ in which these $12$ new edges are the deep blue edges. If we choose other set of $12$ diagonals then we get another map of vertex-type $[3^4, 4^1]$ which is isomorphic to  $\widehat{Y}$ under a reflection of $Y$. Thus, $\widehat{Y}$ is unique up to an isomorphism. Again, it is easy to see that the operations $X \mapsto \widetilde{X}$ and $Y \mapsto \widehat{Y}$ are inverse of each other. Since $24$-vertex map of vertex-type $[3^1, 4^3]$ on $\mathbb{S}^2$ with twelve type $2$ squares is unique, it follows that $24$-vertex map of vertex-type $[3^4, 4^1]$ on $\mathbb{S}^2$ is unique up to isomorphism. 

Let $X$ be a $60$-vertex map of vertex-type $[3^4, 5^1]$. As before we have $30$ deep blue edges and by removing, we get (by similar arguments)  a $60$-vertex map $\widehat{X}$ of vertex-type $[3^1, 4^1, 5^1, 4^1]$. Hence, by Lemma \ref{lem3.12}, $\widehat{X}$ is isomorphic to the boundary of small rhombicosidodecahedron. On the other hand, if $Y$ is the boundary of small rhombicosidodecahedron then $Y$ has $30$ squares. We can divide each of $30$ squares into two triangles by adding one diagonal in two ways.  Again, it is not difficult to see that if we chose one diagonal in one square then there exists exactly one  diagonal for each of the other $29$ squares so that these $30$ new diagonals form a perfect matching. And we get a map $\widetilde{Y}$ of vertex-type $[3^4, 5^1]$ in which these $30$ new edges are the deep blue edges. If we choose other set of $30$ diagonals then we get another map  of vertex-type $[3^4, 5^1]$ which is isomorphic to $\widetilde{Y}$ under a reflection of $Y$. Thus, $\widetilde{Y}$ is unique up to an isomorphism. Again, it is easy to see that the operations $X \mapsto \widehat{X}$ and $Y \mapsto \widetilde{Y}$ are inverse of each other. Since, by Lemma  \ref{lem3.12}, $60$-vertex map of vertex-type $[3^1, 4^1, 5^1, 4^1]$ on $\mathbb{S}^2$ is unique, it follows that $60$-vertex map of vertex-type $[3^4, 5^1]$ on $\mathbb{S}^2$ is unique up to isomorphism.
\end{proof}

\begin{lemma}\label{lem3.14} 
Let $K$ be a semi-equivelar spherical map. If the vertex-type of $K$ is $[4^2, n^1]$ for some $n\geq 3$ then $K$ is isomorphic to the boundary of the $2n$-vertex regular prism $P_n$. \linebreak If the vertex-type of $K$ is $[3^3, n^1]$ for some $n\geq 4$ then $K$ is isomorphic to the boundary of the $2n$-vertex antiprism $Q_n$.
\end{lemma}


\begin{proof}
Let the vertex-type of $K$ be $[4^2, n^1]$. Let $u_1 \in V(K)$ and the link-cycle of $u_1$ be $u_2\mbox{-}u_3\mbox{-}\cdots\mbox{-}u_n\mbox{-}b_n\mbox{-}b_1\mbox{-}b_2\mbox{-}u_2$ where faces $u_1u_2 \dots u_n, u_nu_{1}b_{1}b_n,
u_1u_2b_2b_1$ are incident at $u_1$. Then, the link-cycle of $u_2$ would be $u_3\mbox{-}u_4\mbox{-}\cdots\mbox{-}u_n\mbox{-}u_1\mbox{-}b_1\mbox{-}b_2\mbox{-}b_3\mbox{-}u_3$ for some vertex $b_3$. Clearly, $b_3 \not\in \{b_1, b_2, u_1, \dots, u_n\}$. If $b_3 = b_n$ then the edges $b_3u_3, b_3b_2, b_nu_n, b_nb_1$ are incident with $b_3$. This implies that $\deg(b_3)\geq 4$, a contradiction. So, $b_3 \neq b_n$. Continuing this way, we get the link-cycle of $u_j$ is $u_{j+1}\mbox{-}u_{j+2}\mbox{-}\cdots\mbox{-}u_n\mbox{-}u_1\mbox{-}\cdots\mbox{-}u_{j-1}\mbox{-}b_{j-1}\mbox{-}b_{j}\mbox{-}b_{j+1}\mbox{-}u_{j+1}$ for $j \geq 3$ where $b_t \neq b_{\ell}$ and $u_i \neq u_s$ for $\ell \neq t, i \neq s, 1 \leq \ell, t, i, s \leq n$. Hence, the faces of $K$ are $u_1u_2 \dots u_n, b_1b_2 \dots b_n, u_nu_{1}b_{1}b_n, u_iu_{i+1}b_{i+1} b_i$ $(1 \leq i \leq n-1)$.   Then, $K$ is isomorphic to the boundary of the prism $P_n$.

Let the vertex-type of $K$ be $[3^3, n^1]$. Let $u_1 \in V(K)$ and link-cycle of $u_1$ be $u_2\mbox{-}u_3\mbox{-}\cdots\mbox{-}u_n$ $\mbox{-} b_1\mbox{-}b_2\mbox{-}u_2$, where faces $u_1u_2 \dots u_n, b_1u_1u_n, b_1b_2u_1, u_1u_2b_2$ are incident with $u_1$. Then, the link-cycle of $u_2$ would be $u_3\mbox{-}\cdots\mbox{-}u_n\mbox{-}u_1\mbox{-}b_2\mbox{-}b_3\mbox{-}u_3$ for some vertex $b_3$. Then $b_3 \not\in \{u_1, u_2, \dots, u_n, b_2\}$. Also,  $b_3, b_1$ belong in link-cycle of $b_2$. So, $b_3 \neq b_1$. Hence, let the link-cycle of $u_3$ be $u_4\mbox{-}u_5\mbox{-}\cdots\mbox{-}u_n\mbox{-}u_1\mbox{-}u_2\mbox{-}$ $b_3\mbox{-}b_4\mbox{-}u_4$ for some vertex $b_4$. As before, $b_4 \not\in \{u_1, u_2, \dots, u_n, b_3, b_2\}$. If $b_4 = b_1$ then $b_4$ is in the edges $b_4u_4, b_4u_3, b_4b_3, b_4u_n, b_4u_1, b_4b_2$. This implies that $\deg(b_4) \geq 6$,   a contradiction. So, $b_4 \neq b_1$. Continuing this way, we get link-cycle of $u_j = u_{j+1}\mbox{-}u_{j+2} \mbox{-} \cdots\mbox{-}$ $u_n\mbox{-}u_1\mbox{-}\cdots\mbox{-}u_{j-1}\mbox{-}b_{j}\mbox{-}b_{j+1}\mbox{-}u_{j+1}$ where $b_t \neq b_{\ell}$ and $u_i \neq u_s$ for $\ell \neq t, i \neq s, 1 \leq \ell, t, i, s \leq n$. Clearly, the cycle $b_1\mbox{-}\cdots\mbox{-}b_n\mbox{-}b_1$ is the third face containing $b_j$ for each $j$. So, the faces of $K$ are $u_1u_2 \dots u_n, b_1b_2 \dots b_n, u_nb_nb_{1}, u_1b_{1}u_{n}, u_ib_ib_{i+1}, u_ib_{i+1}u_{i+1}$ $(1 \leq i \leq n-1)$.
Then, $K$ is isomorphic to the boundary of the antiprism $Q_n$.
\end{proof}

\begin{proof}[Proof of Theorem \ref{thm:s2}]
The result follows from Corollary \ref{cor6}  and  Lemmas \ref{lem3.4}, \ref{lem3.11}, \ref{lem3.12}, \ref{lem3.13}, \ref{lem3.14}.
\end{proof}

\begin{definition} \label{def:antipodal}
{\rm An {\em involution}  of a map is an automorphism of order 2. An involution $\rho$ of a map $X$ is called an {\em antipodal automorphism} (or {\em antipodal symmetry}) if no vertex, edge or face of $X$ is fixed by $\rho$. A map $X$ is called {\em centrally symmetric} if $X$ has  an antipodal symmetry.  
}
\end{definition}

We know that the antipodal map of the snub cube (respectively, of the snub dodecahedron) is not an automorphisms of the boundary of the snub cube (respectively, of the snub dodecahedron) (\cite{Cox1940}). We need the following stronger result for the proof of Theorem \ref{thm:rp2}. 

\begin{lemma}\label{lem:new}  If $X$ is the boundary of the snub cube or the boundary of the snub dodecahedron then $X$ does not have any antipodal symmetry. 
\end{lemma}


\begin{figure}[ht]

\tikzstyle{ver}=[]
\tikzstyle{vert}=[circle, draw, fill=black!100, inner sep=0pt, minimum width=4pt]
\tikzstyle{vertex}=[circle, draw, fill=black!00, inner sep=0pt, minimum width=4pt]
\tikzstyle{edge} = [draw,thick,-]
\centering

\begin{tikzpicture}[scale=0.32]

\draw (0, 0) -- (18,0) -- (18,18)-- (0,18) -- (0,0)-- (9,2) -- (18,0)-- (16,9)-- (18,18) -- (9,16)-- (0,18)-- (2,9) -- (0,0) -- (4,4)-- (9,2) -- (14,4)-- (16,9)-- (14,14) -- (9,16)-- (4,14)-- (2,9) -- (4,4)-- (6,6)-- (9,5) -- (12,6)-- (13,9)-- (12,12) -- (9,13)-- (6,12)-- (5,9) -- (6,6)-- (9,7)-- (11,9) -- (9,11)-- (7,9)-- (9,7) -- (9,2);  

\draw (18, 18) -- (14,14) -- (9,13) -- (9,16); \draw (0, 18) -- (4,14) -- (5,9) -- (2,9); 
\draw (18,0) -- (14,4) -- (13,9) -- (16,9) -- (11,9); \draw (14, 4) -- (12,6) -- (11,9) -- (12,12)--(14,14); 
\draw (4, 14) -- (6,12) -- (9,11) -- (9,13) -- (12,12)--(9,11); 
\draw (5, 9) -- (7,9) -- (6,6); \draw (4, 4) -- (9,5) -- (9,7) -- (12,6); 
\draw (7,9) -- (6,12);

 \node[ver] () at (-.6, 17.7) {$1$};  \node[ver] () at (18.7, 17.7) {$2$}; 
  \node[ver] () at (-.6, 0.3) {$4$};  \node[ver] () at (18.7, 0.3) {$3$}; 
 \node[ver] () at (9, 16.7) {$5$};  \node[ver] () at (4.1, 14.7) {$6$}; 
  \node[ver] () at (1.3, 8.9) {$7$};  \node[ver] () at (4.1, 3.3) {$8$}; 
 \node[ver] () at (9, 1.3) {$9$}; \node[ver] () at (13.9, 3.2) {$0$}; 
\node[ver] () at (16.8, 8.9) {$u$}; \node[ver] () at (13.8, 14.8) {$v$}; 
  
 \node[ver] () at (9, 7.7) {$a$}; \node[ver] () at (7.8, 9) {$b$};
 \node[ver] () at (9, 10.3) {$c$};  \node[ver] () at (10.2, 9) {$d$}; 
 \node[ver] () at (5.4, 6) {$e$}; \node[ver] () at (9.45, 4.4) {$f$}; 
\node[ver] () at (11.9, 5.4) {$g$}; \node[ver] () at (13.6, 9.7) {$h$};

\node[ver] () at (12.6, 12) {$i$}; \node[ver] () at (8.5, 13.5) {$j$}; 
\node[ver] () at (6.1, 12.8) {$k$}; \node[ver] () at (4.6, 8.4) {$\ell$};

\end{tikzpicture} 

\caption{{Projection of the boundary of the snub cube on the square 1234}} \label{fig:snub}
\end{figure}


\begin{proof} 
Let $X$ be the boundary of the snub cube. Assume that $X$ is as in Fig. \ref{fig:snub}. Suppose there exists  an antipodal automorphism $\rho$ of $X$. Since no edge is fixed by $\rho$, $x\rho(x)$ is a non-edge 
for each vertex $x$. Also, $\rho$ sends a square to another square. If $\rho(1234)=56kj$ then $\rho(1)=j$ or $k$. If $\rho(1)=j$ then $\rho(1\mbox{-}5\mbox{-}j)$ must be $j\mbox{-}5\mbox{-}1$ and hence $\rho(5)=5$, a contradiction. So, $\rho(1)\neq j$. Similarly,  $\rho(1)\neq k$. Thus, $\rho(1234)\neq 56kj$. Similarly, $\rho(1234)\neq 78el, 90gf, uvih$. So, $\rho(1234)=abcd$. If $\rho(1)=c$ then $\rho(\{2,4\})=\{b,d\}$ and hence $\rho(\{5,7\})=\{k,i\}$. This implies that $\rho(6)=j$ and hence $\rho(56kj)=56kj$, a contradiction. So, $\rho(1)\neq c$. If $\rho(1)=b$ then, by the same argument, $\rho(6)=\ell$ and hence $\rho(6\ell)=6\ell$, a contradiction. So, $\rho(1)\neq b$. Similarly, $\rho(2)\neq c, d$. So, $\rho((1,2))=(a,b)$ or $(d,a)$. In the first case,  $\rho(5)=e$ and hence $\rho(56kj)=78e\ell$. Then $\rho(156) = axy$, where $xy$ is an edge of $78e\ell$ and hence $\rho(156)$ is not a triangle, a contradiction. In the second case,  $\rho(5)=g$ and hence  $\rho(56kj)=gf90$. Then, by the same argument,  $\rho(156)$ is not a triangle, a contradiction.  This shows that $\rho(1234)\neq abcd$. Thus, there does not exist any antipodal automorphism of the boundary of the snub cube. By similar arguments, one can show that there does not exist any antipodal automorphism of the boundary of the snub dodecahedron. 
\end{proof} 

\begin{lemma}\label{lem3.15}  Let $X$ be the boundary of the pseudorhombicuboctahedron. Then 
$(a)$ $X$ is not a vertex-transitive map, and $(b)$ $X$  does not have antipodal symmetry. 
\end{lemma}

\begin{proof} As before, for $2\leq i\leq 4$, let's call a square {\em type $i$} if it intersects $i$ other squares on edges. Then there are two type 4 squares, eight type 3 squares and eight type 2 squares.   Let $u$ be a vertex of a type 4 square (there are 8 such vertices) and $v$ be a vertex which is not in any type 4 square (there are 16 such vertices). If  $\varphi$ is  an automorphism, then the image of a type 4 square under $\varphi$ is again a type 4 square. Therefore, $\varphi(u)$ cannot be $v$. Thus, $X$  is not a vertex-transitive map. This proves part $(a)$. 

Consider $X$ as $A \cup C$ as in Fig. \ref{fig:5}. If possible, let $\rho$ be an antipodal automorphism of $X$. Then $\rho$ sends all the eight type $3$ squares to themselves. And $\rho$ interchanges two type $4$ squares. These imply that $\rho(\{u_5, \dots, u_{12}\}) = \{u_{13}, \dots, u_{20}\}$. Since the edge $u_5u_6$ is in the intersection of two squares, $\rho(u_5u_6)$ is $u_{14}u_{15}$, $u_{16}u_{17}, u_{18}u_{19}$ or $u_{20}u_{13}$. Assume $\rho(u_5u_6) = u_{16}u_{17}$. Since $\rho$ is an antipodal automorphism, this implies that 
\begin{align*}
\rho = (u_5, u_{16})(u_6, u_{17})(u_7, u_{18})(u_8, u_{19})(u_9, u_{20})(u_{10}, u_{13})(u_{11}, u_{14})(u_{12}, u_{15})\cdots.
\end{align*}
Then $\rho(u_5u_6u_{14}u_{13}) = u_{16}u_{17}u_{11}u_{10}$. This is a contradiction since $u_5u_6u_{14}u_{13}$ is a face but $u_{16}u_{17}u_{11}u_{10}$ is not a face of $X$. Similarly, we get contradictions in other three cases. This proves part $(b)$. 
\end{proof} 

\begin{proof}[Proof of Corollary \ref{cor:s2vt}] 
From Lemma \ref{lem3.15} $(a)$, we know that the  boundary of the pseudorhombicuboctahedron is not a vertex-transitive map. 

From the results in \cite{Ba1991}, it follows that all the maps on $\mathbb{S}^2$ other than the boundary of the pseudorhombicuboctahedron mentioned in Theorem \ref{thm:s2}
are vertex-transitive.  The result now follows by Theorem \ref{thm:s2}.
\end{proof}

\begin{proof}[Proof of Theorem \ref{thm:rp2}] 
Let the  number of vertices in $Y$ be $n$ and the vertex-type of $Y$ be $[p_1^{n_1}, p_2^{n_2},$ $\dots, p_{\ell}^{n_{\ell}}]$.
Since $\mathbb{S}^2$ is a $2$-fold cover of $\mathbb{RP}^2$, by pulling back, we get a $2n$-vertex map of vertex-type $[p_1^{n_1}, \dots, p_{\ell}^{n_{\ell}}]$ on $\mathbb{S}^2$. Hence, by Corollary \ref{cor6}, $(n, [p_1^{n_1}, \dots, p_{\ell}^{n_{\ell}}]) = (2, [3^3])$, $(3, [3^4])$, $(4, [4^3]),$ $(6, [3^5]),$ $ (10, [5^3]),$ $(30, [3^4, 5^1]),$ $(12, [3^4, 4^1]),$ $(15, [3^1, 5^1, 3^1, 5^1]),$ $(6, [3^1, 4^1, 3^1, 4^1]),$ $(30, [3^1,  4^1, 5^1, 4^1])$, $(12, [3^1, 4^3])$, $(30,[5^1, 6^2]),$ $(24, [4^1, 6^1, 8^1]),$ 
$(12, [4^1, 6^2])$, 
$(60, [4^1, 6^1, 10^1]),$  $(6, [3^1,$ $ 6^2]),$ $(12,[3^1,$ $ 8^2]),$ $(30, [3^1, 10^2])$, $(3, [3^1, 4^2])$, $(r, [4^2, r^1])$ for some $r \geq 5$ or $(s, [3^3, s^1])$ for some $s \geq 4$.

Clearly, $(n, [p_1^{n_1}, \dots, p_{\ell}^{n_{\ell}}]) = (2, [3^3])$,  $(3, [3^4])$,  $(4, [4^3])$,  $(6, [3^1, 4^1, 3^1, 4^1])$, $(12, [3^1, 8^2])$,  $(6, [3^1, 6^2])$ or $(3, [3^1, 4^2])$ is not possible.

Now, assume $(n, [p_1^{n_1}, \dots, p_{\ell}^{n_{\ell}}]) = (12, [3^4, 4^1])$. Then the $2$-fold cover $\widetilde{Y}$ of $Y$ is a $24$-vertex map of vertex-type $[3^4, 4^1]$ on $\mathbb{S}^2$ with an antipodal symmetry. From Lemma \ref{lem3.12}, $\widetilde{Y}$ is the boundary of snub cube. This is not possible by Lemma \ref{lem:new}. Thus, $(n, [p_1^{n_1}, \dots, p_{\ell}^{n_{\ell}}]) \neq (12, [3^4, 4^1])$.
Similarly, $(n, [p_1^{n_1}, \dots, p_{\ell}^{n_{\ell}}]) \neq (30, [3^4, 5^1])$.

If $[p_1^{n_1}, \dots, p_{\ell}^{n_{\ell}}] = [4^2, n^1]$ for some $n \geq 5$ then $Y$ has a unique face $\alpha$ with $n$ vertices and hence contains all the vertices. Thus, all the vertices of a square $\beta$ are in $\alpha$. Then intersection of $\alpha$ and $\beta$ cannot be a vertex or an edge, a contradiction. Thus, $(n, [p_1^{n_1}, \dots, p_{\ell}^{n_{\ell}}]) \neq (n, [4^2, n^1])$ for any $n \geq 5$.
Similarly $(n, [p_1^{n_1}, \dots, p_{\ell}^{n_{\ell}}]) \neq (n, [3^3, n^1])$ for any $n \geq 4$.
This proves the first part.

Let $(n, [p_1^{n_1}, \dots, p_{\ell}^{n_{\ell}}]) =  (6, [3^5])$. Then $Y$ is a $6$-vertex triangulation of $\mathbb{RP}^2$. It is known (cf. \cite{Da1999}) that there exists unique such triangulation.

Let $(n, [p_1^{n_1}, \dots, p_{\ell}^{n_{\ell}}]) =  (12, [3^1, 4^3])$ then the $2$-fold cover $\widetilde{Y}$ of $Y$ is a $24$-vertex map of vertex-type $[3^1, 4^3]$ on $\mathbb{S}^2$ with an antipodal symmetry. Therefore, by Lemma \ref{lem3.15}, $\widetilde{Y}$ is isomorphic to the boundary of small rhombicuboctahedron. Hence $Y$ is unique.

If $(n, [p_1^{n_1}, \dots, p_{\ell}^{n_{\ell}}]) = (10, [5^3])$ then the $2$-fold cover $\widetilde{Y}$ of $Y$ is a $20$-vertex map of vertex-type $[5^3]$ on $\mathbb{S}^2$. Hence, by Lemma \ref{lem3.4}, $\widetilde{Y}$ is the boundary of the dodecahedron. By identifying antipodal vertices of $\widetilde{Y}$ we get $Y$. Since $\widetilde{Y}$ is unique, $Y$ is unique.  Similarly, for each $(n, [p_1^{n_1}, \dots, p_{\ell}^{n_{\ell}}]) = (24, [4^1, 6^1, 8^1]), (12, [4^1, 6^2]),$ $(30, [5^1, 6^2]),$ $(30, [3^1, 10^2]),$ $(30, [3^1, 4^1, 5^1, 4^1]),(15, [3^1, 5^1, 3^1, 5^1])$ or $(60, [4^1, 6^1, 10^1])$, there is a unique $n$-vertex map of vertex-type $[p_1^{n_1}, $ $\dots, p_{\ell}^{n_{\ell}}]$ on $\mathbb{RP}^2$.
This proves the last part.
\end{proof}

\begin{proof}[Proof of Corollary \ref{cor:tiling}]
If $P$ is one of the  polytopes  mentioned in Section \ref{sec:example}  then (i) the centre of $P$ is $(0,0,0)$, (ii)  all the vertices of $P$ are points on the unit 2-sphere with centre $(0,0,0)$, (iii) each 2-face of $P$ is a regular polygon, and (iv)  lengths of all the edges of $P$ are same. Let $\partial P$ denote the boundary complex of $P$. Let $\eta : \mathbb{R}^3\setminus\{(0,0,0)\} \to \mathbb{S}^2$ be the mapping given by $\eta(x_1, x_2, x_3) = (x_1^2+x_2^2+x_3^2)^{-\frac{1}{2}}(x_1,x_2,x_3)$. Then $\eta(\partial P)$ is a map on $\mathbb{S}^{\hspace{.2mm}2}$. Since $P$ satisfies the properties (i) - (iv), this map is a semi-regular tiling of $\mathbb{S}^2$. (For example, the image of each edge of $P$ under $\eta$ is a great circle arc on $\mathbb{S}^{\hspace{.2mm}2}$ and hence a geodesic.) Part $(a)$ now follows from Theorem \ref{thm:s2}. 

Since $\mathbb{S}^2$ is the universal cover of $\mathbb{RP}^2$, part $(b)$ follows from part $(a)$. 
(More explicitly, let $Y$ be a semi-equivelar map on $\mathbb{RP}^2$. By pulling back to the universal cover, we get a semi-equivelar map $\widetilde{Y}$ (of same vertex-type) on $\mathbb{S}^2$. By part $(a)$, up to an isomorphism, $\widetilde{Y}$ is a semi-regular tiling of $\mathbb{S}^2$ and $Y$ is obtained from $\widetilde{Y}$ as quotient by an order 2 automorphism of $\widetilde{Y}$. This implies that, up to an isomorphism,  $Y$ is a semi-regular tiling on $\mathbb{RP}^{\hspace{.2mm}2}$.)  
\end{proof}

\begin{proof}[Proof of Corollary \ref{cor:rp2vt}]
Since all the maps on $\mathbb{RP}^{\hspace{.2mm}2}$ mentioned in Theorem \ref{thm:rp2} are vertex-transitive (\cite{Ba1991}), the result follows by Theorem \ref{thm:rp2}.
\end{proof}


\medskip

\noindent {\bf Acknowledgements:}
The first author is supported by SERB, DST (Grant No.\,MTR/2017\!/ 000410) and the UGC Centre for Advanced Studies. A part of this work was done while the first author was in residence at the MSRI in Berkeley, California, during the Fall 2017 semester. This visit at MSRI was supported by the National Science Foundation under Grant No.~- 1440140. During a part of this work, the second author was  a UGC - Dr. D. S. Kothari PDF (Award No.~- F.4-2/2006 (BSR)/MA/16-17/0012) at IISc, Bangalore. 
The authors thank the anonymous referees for several useful comments. 


{\small

}

\end{document}